\documentclass[12pt,reqno]{amsart}
\usepackage{amssymb,amsmath, float}

\usepackage[mathscr]{euscript}
\usepackage{enumerate, verbatim, url, color, xcolor, ulem}
\newtheorem{proposition}{Proposition}[section]
\newtheorem{theorem}[proposition]{Theorem}
\newtheorem{corollary}[proposition]{Corollary}
\newtheorem{lemma}[proposition]{Lemma}
\theoremstyle{definition}
\newtheorem{remark}[proposition]{Remark}

\newcommand{\calO}{\mathcal{O}}

\newcommand{\R}{\mathbb{R}}
\newcommand{\C}{\mathbb{C}}
\newcommand{\F}{\mathbb{F}}
\renewcommand{\P}{\mathbb{P}}

\newcommand{\co}{\mathcal{O}}
\newcommand{\Cl}{{\text {\rm Cl}}}
\newcommand{\Tr}{{\text {\rm Tr}}}

\newcommand{\Q}{\mathbb{Q}}

\newcommand{\Z}{\mathbb{Z}}

\newcommand{\SL}{{\text {\rm SL}}}
\newcommand{\GL}{{\text {\rm GL}}}
\newcommand{\textmod}{{\text {\rm mod}}}

\newcommand{\Disc}{{\text {\rm Disc}}}
\newcommand{\Stab}{{\text {\rm Stab}}}
\newcommand{\ct}{{\text {\rm ct}}}

\newcommand{\Res}{\textnormal{Res}}
\newcommand{\calS}{\mathcal{S}}

\newcommand{\Aut}{\textnormal{Aut}}

\newcommand{\textdivblahblah}{\textnormal{div}}
\newcommand{\textnmax}{\textnormal{nmax}}

\newcommand{\twtw}[4]{\left(\begin{smallmatrix}#1&#2\\#3&#4\\\end{smallmatrix}\right)}

\begin{document}
\title[Improved Davenport--Heilbronn]{Improved Error Estimates for the Davenport--Heilbronn Theorems}

\author{Manjul Bhargava}
\address{Department of Mathematics, Princeton University, Princeton, NJ, USA}
\email{bhargava@math.princeton.edu}

\author{Takashi Taniguchi}
\address{Department of Mathematics, Graduate School of Science, Kobe University, Kobe, Japan}
\email{tani@math.kobe-u.ac.jp}

\author{Frank Thorne}
\address{Department of Mathematics, University of South Carolina, Columbia, SC, USA}
\email{thorne@math.sc.edu}

\begin{abstract}
We improve the error terms in the Davenport--Heilbronn theorems on counting cubic fields to $O(X^{2/3 +\epsilon})$.
This improves on separate and independent results of the authors and Shankar and Tsimerman~\cite{BST, TT_rc}.
The present paper uses the analytic theory of Shintani zeta functions, and streamlines and simplifies the proof
relative to \cite{TT_rc}. We also give a second proof
that uses a ``discriminant-reducing identity'' from \cite{BST} and translates it into the language of zeta functions.
We also provide a version of our theorem that counts cubic fields satisfying an arbitrary finite set of local conditions, or even suitable infinite sets of local conditions, where the dependence of the error term
on these conditions is described explicitly and significantly improves \cite{BST, TT_rc}. 
As we explain, these results lead to 
quantitative improvements in various arithmetic applications.
\end{abstract}

\maketitle

\vspace{-.3in}
\section{Introduction}
The purpose of this paper is to improve the error terms in the Davenport--Heilbronn theorems to $O(X^{2/3 + \epsilon})$:

\begin{theorem}\label{thm_rc}
Let $N_3^{\pm}(X)$ denote the number of isomorphism classes of cubic fields $F$ satisfying $0 < \pm \Disc(F) < X$. Then 
\begin{equation}\label{eqn_main_thm}
N_3^{\pm} (X) = C^{\pm} \frac{1}{12 \zeta(3)} X + K^{\pm} 
\frac{4 \zeta(1/3)}{5 \Gamma(2/3)^3 \zeta(5/3)} X^{5/6} + O(X^{2/3} (\log X)^{2.09}),
\end{equation}
where $C^+ = 1$, $C^- = 3$, 
$K^+ = 1$, 
and 
$K^- = \sqrt{3}$.
\end{theorem}

\begin{theorem}\label{thm_rc_torsion}
We have, for any $\epsilon > 0$, that
\begin{equation}\label{eqn_torsion1}
\sum_{0 < \pm D < X} \# \Cl(\Q(\sqrt{D}))[3] = \frac{3 + C^{\pm}}{\pi^2} X + K^{\pm}
\frac{8 \zeta(1/3)}{5 \Gamma(2/3)^3} \prod_p \bigg(1 - \frac{p^{1/3} + 1}{p (p + 1)} \bigg)
 X^{5/6} + O(X^{2/3 + \epsilon}),
\end{equation}
where the sum ranges over fundamental discriminants $D$, 
the expression $\Cl(\Q(\sqrt{D}))[3]$ denotes the $3$-torsion subgroup of the class group of $\Q(\sqrt{D})$, the product is over all primes $p$,
and the constants $C^\pm$ and $K^\pm$ are as in Theorem~$\ref{thm_rc}$.
\end{theorem}
%Here, $\Cl(\Q(\sqrt{D}))[3]$ denotes the 3-torsion subgroup of the class group of $\Q(\sqrt{D})$.
Theorem~\ref{thm_rc_torsion} may be viewed as a counting theorem for cubic fields whose discriminant is fundamental. Indeed, as subgroups of $\Cl(\Q(\sqrt{D}))$ of index $3$ are in bijection
with cubic fields of discriminant $D$ (see, e.g. \cite[Section 8.1]{BST}), Theorem~\ref{thm_rc_torsion} is equivalent
to 
\begin{equation}\label{eqn_torsion2}
N_{3, \textnormal{fund}}^{\pm} (X) = 
\frac{C^{\pm}}{2 \pi^2} X + K^{\pm}
\frac{4 \zeta(1/3)}{5 \Gamma(2/3)^3} \prod_p \bigg(1 - \frac{p^{1/3} + 1}{p (p + 1)} \bigg)
 X^{5/6} + O(X^{2/3 + \epsilon}),
\end{equation}
where $N_{3, \textnormal{fund}}^{\pm} (X)$
denotes
 the number of isomorphism classes of cubic fields $F$ 
 such that $\Disc(F)$ is fundamental and $0 < \pm \Disc(F) < X$.

In each of Theorems~\ref{thm_rc} and~\ref{thm_rc_torsion}, the first main term is due to Davenport and Heilbronn~\cite{DH}, while the second main term was conjectured 
by Datskovsky and Wright \cite[p.~125]{DW3} and Roberts~\cite{Roberts} and proven in \cite{BST} and \cite{TT_rc}.  The latter works in turn built on the successively improved error terms 
obtained 
in~\cite{DH}, \cite{belabas}, and \cite{BBP}. All of these works, including this one, approached these problems by relating them
to counting certain $\GL_2(\Z)$-orbits of integer-coefficient binary cubic forms (though see also~\cite{hough,zhao,ST} for alternative but related~approaches).
We refer to \cite{BST} for a complete, self-contained account of this connection, and to Section~\ref{sec:overview} 
for a briefer summary of the results we will need. 

Note that the error terms in Theorems~1.1 and 1.2, up to the factors of $X^\epsilon$, match the best that one can obtain for the {\itshape smoothed} versions of the same counts, suggesting that we have perhaps achieved the best result possible using our methods.

%\smallskip
We also obtain a variation of Theorem \ref{thm_rc}
that counts (isomorphism classes of) cubic fields satisfying certain specified sets of {\itshape local conditions}, significantly improving the error terms in the corresponding results of 
\cite{BST, TT_rc} and their dependences on these local conditions. As we will discuss, our earlier results saw
applications  to the arithmetic of cubic fields and their $L$-functions, and our improvements yield
corresponding improvements to a number of these applications.

For a prime $p$, let $\Sigma_p$
denote a set of (isomorphism classes of) \'etale cubic algebras
over $\Q_p$.  
We call $\Sigma_p$ a {\itshape local specification} at~$p$, and say that a cubic field $F$ satisfies
$\Sigma_p$ if $F \otimes_\Q \Q_p \in \Sigma_p$.
We say that $\Sigma_p$ is {\it ordinary} if 
it is equal to either $A_p$, the set of all \'etale cubic algebras over~$\Q_p$, or $A_p'$, the set of all \'etale cubic algebras over~$\Q_p$ that are not totally ramified. 

Note that a cubic {\it splitting type} at a prime $p$ is also a local specification at $p$ for a cubic field. There are five ways in which a prime can split in a cubic field.
For each of the three unramified splitting types at $p$, namely, $(111)$, $(12)$, and $(3)$, there is a unique \'etale cubic algebra over~$\Q_p$ having that splitting type. For each of the two ramified splitting types, $(1^21)$ and $(1^3)$, the number of \'etale cubic algebras over $\Q_p$ having that splitting type depends on~$p$.

Let $\Sigma=(\Sigma_p)_p$ be a collection of local specifications 
such that $\Sigma$ is {\it ordinary at $p$, i.e., $\Sigma_p$ is ordinary,} for all but finitely many primes $p$. 
We wish to asymptotically count the total number of cubic fields $F$
of absolute discriminant less than $X$ that agree with this collection of local
specifications, i.e., $F\otimes\Q_p\in\Sigma_p$ for all $p$. We have the following theorem:

\begin{theorem}\label{thm_lc}
Let $\Sigma=(\Sigma_p)_p$ denote a collection of cubic local specifications ordinary at all but finitely many primes $p$, and let $N_3^{\pm}(X,\Sigma)$ denote the number of isomorphism classes of 
cubic fields $F$ satisfying $0 < \pm \Disc(F) < X$ and $F\otimes\Q_p\in \Sigma_p$ for all $p$.
Then 
\begin{equation}\label{eqn_lc}
N_3^{\pm} (X, \Sigma) = C^{\pm}(\Sigma) \frac{1}{12 \zeta(3)} X + K^{\pm}(\Sigma)
\frac{4 \zeta(1/3)}{5 \Gamma(2/3)^3 \zeta(5/3)} X^{5/6} + E(\Sigma)\,O_\epsilon(X^{2/3+\epsilon}),
\end{equation}
where the constants $C^\pm(\Sigma)$, $K^{\pm}(\Sigma_p)$, and $E(\Sigma)$ 
are defined by 
\begin{equation}\label{eq:mt_constants}
C^{\pm}(\Sigma) := C^{\pm} \prod_p C_p(\Sigma_p), \ \ 
K^{\pm}(\Sigma) := K^{\pm} \prod_p K_p(\Sigma_p), \ \ E(\Sigma):=\prod_p E_p(\Sigma_p),
\end{equation}
\begin{equation}\label{eq:mt_constants2}
C_p(\Sigma_p) := \frac{\text{\footnotesize $\displaystyle\sum_{F\in\Sigma_p}$}
\!\textstyle\frac{1}{\Disc_p(F)}\frac{1}{|\Aut(F)|}}{\text{\footnotesize $\displaystyle\sum_{F\in A_p}$} \!\textstyle\frac{1}{\Disc_p(F)}\frac{1}{|\Aut(F)|} }, \;\;\;
K_p(\Sigma_p) := \frac{\text{\footnotesize $\displaystyle\sum_{F\in\Sigma_p}$} 
\!\textstyle\frac{1}{\Disc_p(F)}\frac{1}{|\Aut(F)|} \text{\footnotesize $\displaystyle\int_{\scriptsize\calO_F\setminus p\calO_F}$}\text{\footnotesize $\!\!\![\calO_F\!:\!\Z_p[x]]^{2/3}dx$}}
{\text{\footnotesize $\displaystyle\sum_{F\in A_p}$} 
\!\textstyle\frac{1}{\Disc_p(F)}\frac{1}{|\Aut(F)|} \text{\footnotesize $\displaystyle\int_{\calO_F\setminus p\calO_F}$}\text{\footnotesize $\!\!\![\calO_F\!:\!\Z_p[x]]^{2/3}dx$}}, 
\end{equation}
and
\begin{equation}
E_p(\Sigma_p):=
\begin{cases}
\,1   & \text{if } \Sigma_p=A_p,\,A_p',\text{ or } A_p-A_p';\\
\,p^{2/3} & \text{if } \Sigma_p \text{ is a union of splitting types and } \Sigma_p\neq A_p,\,A_p',\,A_p-A_p';\\
\,p^{8/3} & \text{otherwise.}
\end{cases}
\end{equation}
\end{theorem}

The expressions for 
$C^{\pm}(\Sigma)$ and $K^{\pm}(\Sigma)$,
in terms of products of local densities of \'etale cubic algebras in $\Sigma_p$, were formulated in \cite[p. 118]{DW3} and \cite[Theorem~7]{BST}, respectively,
and
can be evaluated explicitly for any given $\Sigma_p$.
If $\Sigma_p$ corresponds to one of the five cubic splitting types $\mathcal S_p$, then the quantities $C_p(\calS_p)$, $K_p(\calS_p)$, and $E_p(\calS_p)$ are given by the following~table:

\begin{table}[H]
\begin{center}
\begin{tabular}{|c|c|c|c|c|}\hline
Splitting type $\calS_p$ & Notation & $C_p(\calS_p)$ & $K_p(\calS_p)$  & $E_p(\calS_p)$ \\ \hline\hline
Totally split  & $(111)$ & $\frac{p^2/6}{p^2 + p + 1}$ & $\frac{(1 + p^{-1/3})^3}{6} \cdot \frac{1 - p^{-1/3}}{(1 - p^{-5/3})(1 + p^{-1})}$ & $p^{2/3}$ \\ \hline
Partially split & $(21)$   & $\frac{p^2/2}{p^2 + p + 1}$ & $\frac{(1 + p^{-1/3})(1 + p^{-2/3})}{2} \cdot \frac{1 - p^{-1/3}}{(1 - p^{-5/3})(1 + p^{-1})}$ & $p^{2/3}$ \\ \hline
Inert & $(3)$ & $\frac{p^2/3}{p^2 + p + 1}$ &  $\frac{1 + p^{-1}}{3} \cdot \frac{1 - p^{-1/3}}{(1 - p^{-5/3})(1 + p^{-1})}$ & $p^{2/3}$ \\ \hline
Partially ramified & $(1^21)$ & $\frac{p}{p^2 + p + 1}$ &  $\frac{(1 + p^{-1/3})^2}{p} \cdot \frac{1 - p^{-1/3}}{(1 - p^{-5/3})(1 + p^{-1})}$  & $\:\!\phantom{{}^*}p^{2/3\;\!*}$\\ \hline
Totally ramified& $(1^3)$ & $\frac{1}{p^2 + p + 1}$ &  $\frac{1 + p^{-1/3}}{p^2} \cdot \frac{1 - p^{-1/3}}{(1 - p^{-5/3})(1 + p^{-1})}$ & $\:\!\phantom{{}^*}1^{\:\!*}$ 
\\  \hline
\end{tabular}
\end{center}
\caption{Splitting types, main term constants, and error term constants.} \label{table:error_exponents}
\end{table}\vspace{-.25in}

If a local specification $\Sigma_p$ is a union of such splitting types, then the constants $C_p(\Sigma_p)$ and $K_p(\Sigma_p)$ are the sum of those
for the respective splitting types, while the error constant $E_p(\Sigma_p)$ is the maximum (except in the ordinary cases $\Sigma_p=A_p$ or $A_p'$, where we may take $E_p(\Sigma_p)=1$).
By construction, the constants $C_p(\calS_p)$ and $K_p(\calS_p)$ for the five splitting conditions $\calS_p$ each sum to $1$.

The asterisks indicate error terms that can be improved by averaging, as we will describe shortly.

\medskip
\noindent{\bf Local completions not determined by splitting types.}
To complete the description of Theorem \ref{thm_lc},
we describe the constants
$C_p(\Sigma_p)$, $K_p(\Sigma_p)$, and $E_p(\Sigma_p)$ in the cases that $\Sigma_p$ consists of a single \'etale cubic extension $F$ of $\Q_p$ that is not completely specified by its splitting type $\calS_p$.
\begin{itemize}
\item
If $p \neq 2$, then the  ``partially ramified'' condition $\calS_p$ corresponds to two choices for $F \otimes \Q_p$. If $\Sigma_p$
consists of only one of these, then $C_p(\Sigma_p)=\frac12C_p(\calS_p)$ and $K_p(\Sigma_p)=\frac12K_p(\calS_p)$.
\smallskip\item
If $p \equiv 1 \!\pmod{3}$, then the ``totally ramified'' condition $\calS_p$ corresponds to three choices for $F\otimes\Q_p$. If $\Sigma_p$ consists of only one of these, then
$C_p(\Sigma_p)=\frac13C_p(\calS_p)$ and $K_p(\Sigma_p)=\frac13K_p(\calS_p)$.
\smallskip\item
If $p = 2$, then the  ``partially ramified'' condition $\calS_2$ corresponds to six choices for $F \otimes_\Q \Q_p$, each of which 
is of the form $\Q_2 \times E$ where $E$ is a ramified quadratic extension of~$\Q_2$. Generating polynomials for these six possibilities 
are given in Section~\ref{appendix:count_quadratic}, together with constants $c_2$ giving the proportion of quadratic number fields $F_2$
such that $F_2 \otimes_{\Q} \Q_2 \simeq E$.
If~$\Sigma_2$ consists of only one of these choices, then
$C_2(\Sigma_2)=3c_2C_2(\calS_2)$ and $K_p(\Sigma_p)=3c_2 K_2(\calS_2).$
\smallskip\item
Finally, if $p = 3$, then the ``totally ramified'' condition $\calS_3$ corresponds to nine choices for $F \otimes_\Q \Q_p$; we refer to 
\cite[Section 6.2]{TT_rc} for a list of generating polynomials for these nine extensions and the constants that 
$C_3(\calS_3)$ and $K_3(\calS_3)$ should be multiplied by in the case that $\Sigma_3$ consists only of one of these nine extensions.
\end{itemize}
We may always take $E_p(\Sigma_p) = 8/3$, although with additional work this could be improved.

Theorem \ref{thm_lc} can also be extended to count {\it nonmaximal} cubic rings satisfying suitable local specifications (see Remark \ref{remark:other_split}).

\subsection*{Averaged error terms} As indicated by the asterisks in the table, we may obtain stronger bounds on $E(\Sigma)$ on average when such local conditions are imposed over ranges of primes.
Let~$U$ be any positive integer,
and for each $p \mid U$ let $\Sigma_p$ be an arbitrary local specification.
For each pair $r,t$ of positive squarefree integers such that $(rt,U)=1$, we  complete this to a collection
\smash{$\Sigma^{r,t} = (\Sigma_p)_p$} of local specifications over all $p$ satisfying the following conditions:
\begin{itemize}
\item If $p\mid r$, then \smash{$\Sigma_p$} consists of all partially ramified cubic extensions of $\Q_p$.
\smallskip
\item If $p\mid t$, then \smash{$\Sigma_p$} consists of all totally ramified cubic extensions of $\Q_p$.
\smallskip
\item If $p\nmid Urt$, then \smash{$\Sigma_p$} is ordinary, i.e., \smash{$\Sigma_p=A_p$} or $A_p'$.
\end{itemize}
\begin{theorem}\label{thm_lc_averaged}
For each collection of cubic local specifications $\Sigma=(\Sigma_p)_{p}$, 
let $E(X, \Sigma)$ denote the
error term in estimating  $N_3^{\pm}(X, \Sigma)$ using
\eqref{eqn_lc}. Then, for any $\epsilon>0$, we have 
\begin{equation}\label{eq:lc_averaged}
\sum_{r \leq R} \sum_{t \leq T} |E(X_{r,t},\Sigma^{r,t})| \ll X^{2/3+\epsilon} R^{2/3}\prod_{p \mid U} E_p(\Sigma_p)
\end{equation}
for each $R,T,X>0$ 
and any $X_{r, t}$ with $X_{r, t} \leq X$;
here the sum is over squarefree integers $r$ and $t$ with $(U,rt)=1$.
\end{theorem}
In other words, we may sum over ranges of $r$ and $t$ ``for free''.  Indeed, Theorem \ref{thm_lc} follows immediately from Theorem \ref{thm_lc_averaged} 
by applying the bound on the right-hand side of \eqref{eq:lc_averaged} to each individual summand on the left.

\subsection*{Levels of distribution}
Our results may be interpreted as ``level of distribution'' estimates for cubic
fields with respect to local conditions. A number of related such results were obtained by 
Belabas and Fouvry \cite{BF}, and in general we obtain quantitatively stronger results.

We expect the following two corollaries to perhaps
be the most useful in applications; variations can be deduced in the same way.

\begin{corollary}\label{cor:ld_total}
For each collection of cubic local specifications $\Sigma=(\Sigma_p)_{p}$, 
let $E(X, \Sigma)$ denote again the
error in estimating  $N_3^{\pm}(X, \Sigma)$ using
\eqref{eqn_lc}. Then for each~$\epsilon, A > 0$, we have 
\begin{equation}\label{eq:ld_total}
\sum_{q < X^{1/5 - \epsilon}} \mu^2(q) 
\sum_{\Sigma \!\!\!\!\pmod q} |E(X, \Sigma)| \ll_{\epsilon, A} \frac{X}{(\log X)^A},
\end{equation}
where the inner sum is over all collections $\Sigma = (\Sigma_p)_p$ such that if $p \mid q$, then $\Sigma_p$  corresponds to one of the $32$ subsets of the five splitting types, and if $p \nmid q$, then $\Sigma_p = A_p$.
\end{corollary}

This is immediate from Theorem \ref{thm_lc}: there
are $\ll X^{\epsilon}$ choices of $\Sigma$ for any given $q$, and for each such choice we have
$|E(X, \Sigma)| \ll X^{2/3 + \epsilon} q^{2/3}$; this error may be summed over all $q$ up to~$X^{1/5 - \epsilon}$ and remain within the upper bound of \eqref{eq:ld_total}.
Moreover, if we instead sum over all~collections $\Sigma$ where $\Sigma_p = A_p$ {\it or} $\Sigma_p = A'_p$ for each $p \nmid q$, independently, then Corollary~\ref{cor:ld_total} still holds.

The following corollary follows similarly from Theorem \ref{thm_lc_averaged}. 

\begin{corollary}
Let $N_3^{\pm}(X, q)$ denote the number of isomorphism classes of cubic fields $F$ such that 
$q\mid \Disc(F)$ and $0 < \pm \Disc(F) < X$. 
Let $E(X, q)$ denote the error term in estimating $N_3^{\pm}(X, q)$ using 
\eqref{eqn_lc}. Then, for each $\epsilon, A > 0$, we have 
\begin{equation}
\sum_{q < X^{1/2 - \epsilon}} \mu^2(q) |E(X, q)| \ll_{\epsilon, A} \frac{X}{(\log X)^A}.
\end{equation}
The same holds if one counts only cubic fields having squarefree discriminant.
\end{corollary}

\subsection*{Applications.} Our previous results in \cite{BST, TT_rc} were applied by various authors
to obtain further results concerning the arithmetic of cubic fields and their $L$-functions. In some cases, our improvements lead
to quantitative results there. Here are several examples:

\medskip
1. Let $N_6(X; S_3)$ be the number of $S_3$-sextic fields $L$ with 
$|\Disc(L)| < X$. Asymptotics for $N_6(X; S_3)$ were obtained independently by Belabas and Fouvry~\cite{BF_S3}
and the first author and Wood \cite{BW}. The second and third authors obtained a power
saving error term in \cite{TT_S3}, proving that
\[
N_6(X; S_3) = c X^{1/3} + O(X^{1/3 - \frac{5}{447} + \epsilon})
\]
for an explicit constant $c$.

We can now improve this error term to $O(X^{2/7 + \epsilon})$, only slightly larger than a conjectured secondary term
of order $X^{5/18}$. In \cite[Theorem 2.2]{TT_S3}, we now
have $(\alpha,\beta)=(-1,2/3+\epsilon)$ on average, which means that the choice $Q = X^{1/7}$ is admissible in and just
below \cite[(2.7)]{TT_S3}.

\medskip
2. Let $A$ be an abelian group with minimal
prime divisor of $|A|$ greater than $5$. In \cite{wang},
Wang obtains an asymptotic formula counting
degree $3|A|$ extensions with Galois group 
$S_3 \times A$, of the form
\[
N(S_3 \times A) = C_1 X^{1/|A|} +
C_2 X^{5/6|A|} + O(X^{5/6|A| - \delta}),
\]
for explicit constants $C_1$, $C_2$, and $\delta$.

Her formula
for 
the discriminant of an $S_3 \times A$-extension
is determined by that of the $S_3$- and
$A$-subextensions, and the most subtle part
is an arithmetic factor divisible by
those primes ramified in the $S_3$- and 
$A$-subextensions. This makes it necessary to
count the number of $S_3$-cubic fields
with prescribed ramification behavior.

As of this writing, Wang's work relies on a version of Theorem \ref{thm_lc} in an earlier draft of this paper, in which we had obtained results stronger than those of \cite{TT_rc},
but weaker than Theorem~\ref{thm_lc_averaged}. She informs us that our improvements may allow her to relax the condition on the minimal prime divisor of $|A|$.

\medskip
3. In \cite{MT}, McGown and Tucker studied the statistics of {\itshape genus numbers} of cubic fields.
Following \cite{MT}, the genus field of a number field $F$ is defined to be the maximal extension
$F'/F$ that is (a) unramified at all finite primes and (b) a compositum of the form $Fk^*$ with $k^*$ absolutely abelian.
The {\itshape genus number} $g_F$ is then defined to be $[F' : F]$.

Let $N^{\pm}_{\textnormal{genus}}(X)$ denote the count of cubic fields $F$ with $0 < \pm \Disc(F) < X$ such that~$g_F = 1$. (Equivalently, this counts cubic fields for which there is no nontrivial extension satisfying (a) and (b) above.)
Then McGown and Tucker prove that
\[
N^{\pm}_{\textnormal{genus}}(X) = \frac{29C^\pm}{324 \zeta(2)} \prod_{p \equiv 2 \pmod 3} \left( 1 + \frac{1}{p(p + 1)} \right) X + O(X^{16/17 + \epsilon}).
\]
By using Theorem \ref{thm_lc_averaged} instead of the results of \cite{TT_rc} in their proof, the error term above can be improved to $O(X^{2/3 + \epsilon})$,
with (as expected) a secondary term of order $X^{5/6}$.

\medskip
4. Cho and Kim \cite{CK},
Shankar, S\"odergren, and Templier \cite{SST}, and Yang \cite{yang} all obtained one-level density results for the Artin $L$-functions associated
to cubic fields. Such results take the
following form. For each cubic number field $F$, let $L(s, \rho_F) = \zeta_F(s)/\zeta(s)$
be its associated Artin $L$-function. In \cite[Section 2]{SST}, it is proved that
\begin{equation}\label{eq:sst}
\lim_{X \rightarrow \infty} \frac{1}{N_3(X)} \sum_{|\Disc(F)| < X} \sum_{L(\frac{1}{2} + i \gamma_F, \rho_F) = 0} 
f \left( \frac{\gamma_F \mathscr{L}}{2 \pi}\right) = \widehat{f}(0) - \frac{f(0)}{2},
\end{equation}
where: the inner sum is over nontrivial zeroes $\frac{1}{2} + i \gamma_F$ of $L(s, \rho_F)$, not necessarily assumed to
lie on the $\frac{1}{2}$-line; $\mathscr{L} \approx \log(X)$ is a normalizing factor; and 
$f$ is a Paley-Wiener function whose Fourier transform is smooth and supported in 
$(-\frac{4}{41}, \frac{4}{41})$.

Their proof relies on the ``explicit
formula'' relating the left side of
\eqref{eq:sst} to the prime power coefficients of the
$L$-functions $L(s, \rho_F)$. These coefficients
are determined by the splitting behavior of primes
in these cubic fields $F$; therefore, our results
imply asymptotic density results 
(with power-saving error terms) for the
values of these coefficients on average over~$F$.

These authors relied on the results 
of \cite{BBP, TT_rc} which are improved here.
Applying Theorem~\ref{thm_lc_averaged} in their proof, we immediately improve their range of support on the Fourier transform 
to $(- \frac27, \frac27)$. 
(See also \cite{CK} and \cite{yang} for very similar results, with \cite{CK} also allowing for twisting by an automorphic form.)

In more recent work, Shankar, S\"odergren, and Templier \cite[Theorem 5]{SST_central} improved the range of support to $(- \frac25, \frac25)$, provided
that a local specification is imposed requiring any fixed prime to be inert. Instead of relying on \cite{BBP, TT_rc}, they developed (independently of the
present paper) a smooth variation of Davenport-Heilbronn, with error terms comparable to ours. They also observed that, when considering Artin characters
in lieu of the individual splitting conditions in Table \ref{table:error_exponents}, some cancellation occurs in the Fourier transform and better error terms
can be obtained. Incorporating this observation into the proof in \cite{SST} or \cite{SST_central}, we are also able to obtain $(- \frac25, \frac25)$, and without any
required local specification.

\medskip
5. Using the secondary term in Theorem \ref{thm_lc} and in \cite{BST, TT_rc}, 
Cho, Fiorilli, Lee, and S\"odergren~\cite{CFLS} refined the latter one-level density estimates to exhibit
a secondary term in that counting function as well.
Moreover, they proved that even a tiny 
improvement to our error term in Theorem~\ref{thm_rc} would lead to a larger than 
expected omega result in the Ratios Conjecture of Conrey, Farmer, and Zirnbauer \cite{CFZ}.
They also proved that improving 
the error in Theorem \ref{thm_lc} to $O(p^\omega X^{\theta})$ for all $p$ and $X$, for any $\omega$ and $\theta$ with $\omega + \theta < \frac{1}{2}$,
would contradict the Generalized Riemann Hypothesis. This is thus 
a conditional omega result establishing a limitation on the strength of the error estimates that can actually~hold.

\subsection*{Method of proof}
Our proofs apply the theory of {\itshape Shintani zeta functions}, which were introduced 
by Sato and Shintani in their landmark works~\cite{SS, shintani}.  The theory was extended to the adelic setting by Datskovsky and Wright~\cite{DW1, DW2}, and developed further by the second and third authors in \cite{TT_L}.  In \cite{TT_rc}, the theory was utilized to obtain \eqref{eqn_main_thm} and \eqref{eqn_torsion1} with error terms of $O(X^{7/9 + \epsilon})$ and~$O(X^{18/23 + \epsilon})$, respectively.

In this paper, we will in fact give two variants of the improved estimates in Theorem~\ref{thm_rc}:
\begin{itemize}
\item Our simplest proof applies Landau's method for estimating partial sums of Dirichlet series having
analytic continuation and a functional equation. While this was also utilized in \cite{TT_rc}, we will deploy Landau's method in a more effective manner that keeps track of the dependence on the densities of the local conditions being considered (equivalently, the dependence on the residues of the associated Shintani zeta functions), and also takes advantage of the average behavior of the Fourier transforms of the local conditions. 
This not only substantially lowers the error terms,
but also greatly simplifies the~proof. 
The method enables us to prove
Theorems~\ref{thm_rc} and~\ref{thm_rc_torsion} simultaneously, 
with error terms of $O(X^{2/3 + \epsilon})$.
This is carried out in Section~\ref{sec:direct_proof}.

\smallskip
\item We also give a proof of Theorem \ref{thm_rc} using a ``discriminant-reducing identity''.
Such an identity was used in the first author's work with Shankar and Tsimerman \cite{BST}, obtaining \eqref{eqn_main_thm}
with an error term of $O(X^{13/16 + \epsilon})$. We translate the identity into the language of Shintani zeta functions, and 
deploy it within our use of Landau's method. 
We also give a more precise treatment of the error term in this proof,
replacing  the $O(X^{\epsilon})$ with $O( (\log X)^\alpha)$
for any $\alpha > - \frac12 + \frac{5}{3^{3/5}} = 2.0864\dots$.
This is carried out in Section \ref{sec:disc_reduce_1}.

\end{itemize}

The second proof variant is in some ways more complicated, since in the end we must resort
to Landau's method anyhow. However, it avoids several messy computations
(most of which were carried out in \cite{TT_L}), and it showcases how such a  
``discriminant-reducing identity'' may be used. 
Indeed, similar arguments could potentially be applied to prove Theorem~\ref{thm_rc_torsion} using the recent discriminant-reducing identities due to 
O'Dorney \cite{odorney, odorney2} (which in turn build on the identities of Ohno~\cite{Ohno} and Nakagawa~\cite{nakagawa}), though we do not pursue that here. We suspect that 
these methods may have further applications as well.

In the interest of brevity, we will describe the first proof in detail, and then will give a full account only of those elements of the
second proof that are new. In particular, the second proof leads to a different (but equivalent) computation of the main terms
in our main results, the details of which we will omit.

Our proof of Theorems \ref{thm_lc} and \ref{thm_lc_averaged}, which count cubic fields with local conditions,
is based on our first proof method, and we explain the details in Section \ref{sec:lc}. We go to some efforts to
optimize the dependence of the error terms on these local conditions. For example, with~$U = 1$, the error term
of $O(X^{2/3 + \epsilon} R^{2/3})$ in Theorem \ref{thm_lc_averaged} significantly improves upon the error of
$O(X^{7/9 + \epsilon} (RT)^{25/9})$ in~\cite[Theorem 1.3]{TT_rc}, while also
holding in greater generality. This quantitative improvement is reflected in the substantial corresponding quantitative improvements
in the various applications discussed above.
The proof applies exponential sum formulas of 
Mori \cite{Mori} and the second and third authors
\cite{TT_L} together with a further averaging technique to treat the 
error terms as effectively as possible.

There are a variety of results similar to the Davenport--Heilbronn
theorems in the literature, such as the first author's work counting quartic \cite{B_quartic} and quintic \cite{B_quintic} fields,
together with a number of applications, and we hope that the framework
formulated in this article will provide useful tools
in those situations as well.

%\smallskip
%\blue{{\bf Prospects for further improvement.} 
\subsection*{Prospects for further improvement}{
As mentioned earlier, our error terms in Theorems~1.1 and 1.2 essentially match what one can prove for the {\itshape smoothed}
versions of the same counts: see independent work of Shankar, S\"odergren, and Templier \cite[Theorem 6.11]{SST_central} for such a smooth count, as well as work of Hough \cite[Theorem 1.1]{hough_shape}, smoothing while also twisting by a cusp form. 
The work \cite{SST_central} also includes an analogue of our results with local conditions, with the dependence of the error term matching Table \ref{table:error_exponents} in the cases they treat,
illustrating that we have perhaps achieved the best result possible using our methods.

To improve our error term beyond $O(X^{2/3 + \epsilon})$, it seems necessary to in some way obtain cancellation in exponential sums. 
One classical way of doing so is in the form of a {\itshape subconvexity bound}, which Hough and Lee \cite{HL} have recently proved in $t$ aspect
for the classical Shintani zeta function. If their results could also be obtained for the generalizations described in \eqref{eqn_shintani_def_2}, then there is
some prospect of a (very small) additional power savings in our error term. 

A savings in conductor aspect, either in a subconvexity estimate or to the left of the critical strip, could also likely prove useful. 
For example, in Proposition \ref{lem:ft_eval}, we state upper bounds for the Fourier transform \smash{$\widehat{\Psi_{p^2}}$} of a function
defined on $(\mathbb{Z}/p^2 \mathbb{Z})^4$. The absolute value of this Fourier transform is $O(p^{-7})$ on average, and our 
method is able to fully exploit this~fact.  

In our methods, we currently do not exploit sign cancellations in the value of  $\smash{\widehat{\Psi_{p^2}}}$; without absolute values,
this Fourier transform is $O(p^{-10})$ on average. Exploiting such cancellations appears likely to be quite a challenging problem, and would require modifications within
Landau's method (or a substitute), but may be possible, for example with techniques related to those used in \cite{HL}.

Finally, we expect that our methods should allow for a treatment of cubic fields in arithmetic progressions, building on what was done in \cite{TT_rc};
we anticipate error terms of $O(X^{2/3 + \epsilon} m^{\alpha})$, where $m$ is the modulus of the arithmetic progression and $\alpha$ is a
reasonably small positive constant.
}

\subsection*{Organization of the paper} We begin in Section \ref{sec:overview} by recalling necessary background material
on binary cubic forms and the associated Shintani zeta functions. 
In Section \ref{sec:landau}, we apply the method of Landau, 
as formulated in the work of Lowry-Duda and the second and third authors~\cite{LDTT},
 to estimate partial sums of coefficients of congruence Shintani zeta functions. (This takes the place
of ``Davenport's Lemma'' in geometry-of-numbers approaches.)
We provide some additional
preliminaries regarding reducible rings and uniformity estimates in Section \ref{sec:prelim}.

In Section \ref{sec:direct_proof}, we then give our first proof of Theorems \ref{thm_rc} and \ref{thm_rc_torsion}, with error terms of $O(X^{2/3 + \epsilon})$. 
In Section \ref{sec:lc}, 
we prove Theorem \ref{thm_lc_averaged}, 
and thus Theorem \ref{thm_lc}, 
along similar lines,
though with some additional setup and notation required.
In Section \ref{sec:disc_reduce_1}, we prove our discriminant-reducing identity, and use it to give a second proof of Theorem \ref{thm_rc},
this time with the stated error term of $O(X^{2/3} (\log X)^{2.09})$.

Finally, in Section \ref{appendix:count_quadratic}, we prove a result counting {\itshape quadratic} fields with local conditions, which is also needed in the
proofs of our main theorems. A similar such result was proved by Ellenberg, Pierce, and Wood in \cite{EPW}. We generalize their result slightly by allowing local specifications that are not simply given by local splitting types, and we improve the error term in their work by arranging the relevant sum in a manner that enables an application of the P\'olya--Vinogradov inequality. 

\section{Background: Binary cubic forms and Shintani zeta functions}\label{sec:overview}
Following \cite{DH, BST, TT_rc}, we count cubic fields by means of the 
{\itshape Levi--Delone--Faddeev correspondence},
which relates binary cubic forms to cubic rings, and the {\itshape Davenport--Heilbronn correspondence},
which describes maximality conditions for these rings in terms of congruence conditions
on the cubic forms. After discussing the relationship between cubic fields and cubic rings, we 
recall these two correspondences (and the basic definitions). We then introduce
the Shintani zeta functions which we use to prove our counting theorems.

\subsection{Cubic fields, orders, and rings}\label{subsec:cubic_fields}
A {\itshape cubic ring} (over $\Z$) is a commutative ring that is free of rank 3 as
a $\Z$-module. Its {\itshape discriminant} is the determinant of the trace form
$\langle x, y \rangle = \Tr(xy)$. 
Cubic fields are in bijection with their maximal orders, and more generally so are cubic {\itshape \'etale algebras}
(products of number fields of total degree $3$). The discriminant of a cubic \'etale algebra $F$ is, by definition, equal to the discriminant of its maximal order~$\co_F$. 

Our methods most directly count cubic rings, and a cubic ring $R$ is 
the maximal order in a cubic field if and only if it is:
(i) an integral domain; and (ii) a maximal cubic ring (i.e., there is no other cubic ring $R'$ properly containing $R$).
Note that if $R$ is an integral domain, then it is automatically nondegenerate (i.e., $\Disc(R) \neq 0$).
Maximality may be checked locally: a cubic ring $R$ is maximal if and only if $R \otimes_{\Z} \Z_p$ is maximal 
as a cubic ring over $\Z_p$ for all $p$.

\subsection{Binary cubic forms}\label{subsec:cubic_forms}

The lattice of {\itshape integral binary cubic forms} is defined by
\begin{equation}\label{def_vz}
V(\Z) := \{ a u^3 + b u^2 v + c u v^2 + d v^3 \ : a, b, c, d \in \Z \},
\end{equation}
and the {\itshape discriminant} of
$f(u,v)=a u^3 + b u^2 v + c u v^2 + d v^3\in V(\Z)$
is given by the equation
\begin{equation}\label{eqn_disc_formula}
\Disc(f) = b^2 c^2 - 4 a c^3 - 4 b^3 d - 27 a^2 d^2 + 18 abcd.
\end{equation}
The group $\GL_2(\Z)$ acts on $V(\Z)$ by
\begin{equation}
(\gamma \cdot f)(u, v) = \frac{1}{\det \gamma} f((u, v) \cdot \gamma).
\end{equation}
A cubic form $f$ is {\itshape irreducible} if $f(u, v)$ is irreducible as a polynomial over $\Q$, and
{\itshape nondegenerate} if $\Disc(f) \neq 0$.

The 
correspondence of Levi~\cite{levi} and Delone--Faddeev~\cite{DF},
as further extended by Gan, Gross, and Savin~\cite{GGS}
to include the degenerate case, is as follows:

\begin{theorem}[\cite{levi, DF,GGS}]\label{thm:df} There is a canonical, discriminant-preserving 
bijection between the set of $\GL_2(\Z)$-orbits on $V(\Z)$ and the set of isomorphism classes of cubic rings. Under this
correspondence, irreducible cubic forms correspond to orders in cubic fields, and if a cubic form $f$
corresponds to a cubic ring $R$, then $\Stab_{\GL_2(\Z)}(f)$ is isomorphic to $\Aut(R)$.
\end{theorem}

\noindent
The Davenport--Heilbronn maximality condition is the following:
\begin{proposition}[\cite{DH}\label{def_up}]
Under the Levi--Delone--Faddeev correspondence, a cubic ring $R$ is maximal if and only if any corresponding
cubic form $f$ belongs to the set $U_p \subset V(\Z)$ for all~$p$, defined by the 
following conditions:
\begin{itemize}
\item
the cubic form $f$ is not a multiple of $p$; and \smallskip \item there is no $\GL_2(\Z)$-transformation of
$f(u, v) = a u^3 + b u^2 v + c u v^2 + d v^3$ such that $a$~is~a multiple of $p^2$ and $b$ is a multiple
of $p$.
\end{itemize}
\end{proposition}\noindent We say that a cubic form $f$ is {\it maximal at $p$} if $f \in U_p$.

Davenport and Heilbronn's proof of the main term in \eqref{eqn_main_thm} 
can be summarized as follows. One obtains an
asymptotic formula for the number of cubic rings of bounded discriminant by counting lattice points
in a fundamental domain for the action of $\GL_2(\Z)$ on binary cubic forms $f$ over $\R$, subject to the constraint $|\Disc(f)| < X$.
The fundamental domain may be chosen so that almost all rings that are not integral domains correspond to
forms with $a = 0$, and so these may be excluded from the count.
One then multiplies this asymptotic by the 
product of all the local densities of the sets $U_p$. This yields
a heuristic argument for the main term in \eqref{eqn_main_thm}, and some careful analysis allows one to convert this heuristic into a~proof.

We also recall (e.g., from \cite[Section 8.2]{BST}) the notion of the {\itshape content} of a cubic ring and a binary cubic form.
The content $\ct(R)$ of a cubic ring $R$ is the largest integer $n$ such that  $R = \Z + nR'$ for some cubic ring $R'$; the content 
of a binary cubic form is the gcd of its coefficients. As explained in \cite{BST}, a cubic form $f$ and its corresponding 
cubic ring $R$ have the same content.

We also require the following result on overrings and subrings of a cubic ring having a given squarefree index:
\begin{lemma}\label{lemma_overrings} Let $q$ be a squarefree integer.
\begin{enumerate}[{\rm (i)}]
\item
Let $R$ be a cubic ring that is nonmaximal at each prime divisor of $q$ and whose content is coprime to $q$. Then
$R$ is contained in an overring $R'$ with index $q$.
\item For any ring $R'$, the number of $R$ contained in $R'$ with index $q$ is bounded above by
\[
\prod_{\substack{p \mid q \\ p \nmid \ct(R')}} 3 
\prod_{\substack{p \mid q \\ p \mid \ct(R')}} (p + 1).
\]
\end{enumerate}
\end{lemma}

\begin{proof} This follows from \cite[Lemma 2.4]{BBP} and \cite[Lemma 3.5]{TT_rc}.
\end{proof}

\subsection{Local conditions for rings and forms}\label{subsec:lc}
Let $R$ be a maximal cubic ring
corresponding to a binary cubic form $f$
(well defined up to $\GL_2(\Z)$-equivalence).
We have just seen in Proposition 
\ref{def_up} that, for each prime $p$, the form
$f$ is required to satisfy
certain congruence conditions modulo $p^2$.

Let $f$ be a binary cubic form in $V(\Z)$ with content prime to $p$. The {\itshape splitting type} of $f$ modulo $p$ is the combinatorial data 
describing the number of roots of $f$ in $\P^1(\overline{\F_p})$, together with multiplicities and 
degrees of their fields of definition. For example, we say that $f$ has splitting type
$(1^2 1)$ if it has a double root and a single root in $\P^1(\F_p)$, or $(3)$ if it is irreducible over $\F_p$. The possible splitting types are those listed in Table~\ref{table:error_exponents}.
The basic result is that, for the five
splitting types enumerated in Table
\ref{table:error_exponents}, a cubic field $F$ has the
specified splitting type if and only if the binary cubic form
$f$ corresponding to the maximal order of $F$ has the analogous splitting type,
indicated with the same notation. 

With this in mind, to each splitting type (mod $p$) we may associate
the characteristic function 
$\Phi : V(\Z) \rightarrow \{0, 1\}$
of those
$f \in V(\Z)$ that have that splitting
type and  are in the set $U_p$
of Proposition \ref{def_up}. This function
factors
through $V(\Z/p^2\Z)$ to define a
function 
$\Phi_{p^2} : V(\Z/p^2 \Z) \rightarrow \{0, 1\}$.
Moreover, for the splitting types
$(111), (21), (3)$, this function factors
through $V(\Z/p\Z)$ to define a function
$\Phi_{p} \ : V(\Z/p\Z) \rightarrow \{0, 1\}$.
(Indeed, any $f$ with an unramified splitting
type (mod $p$) satisfies $p \nmid \Disc(f)$, and hence
is automatically in~$U_p$.)

More generally, let $\Sigma_p$ denote a cubic local specification at $p$. 
The algebras in $\Sigma_p$ are also then
all defined by congruence conditions modulo $p^a$
for some $a>0$. (We may always take $a \leq 2$
when $p > 3$.)
In this case, too, we may associate to $\Sigma_p$ the characteristic function $\Phi_{p^a}\ : V(\Z/p^a\Z) \rightarrow \{0, 1\}$ of those $f\in U_p$ that correspond to the maximal orders of
cubic algebras in $\Sigma_p$.

Let $\Sigma = (\Sigma_p)_p$ denote any collection of cubic local specifications.
We define the 
{\itshape conductor}
$M$ of $\Sigma$
by $M := \prod_{p} p^a$, where the product is over all
primes $p$ for which $\Sigma_p$ is not ordinary and $a$ is the minimal
integer for which the associated characteristic
function factors through $V(\Z/p^a \Z)$.
We define $\Phi_M := \otimes_{p \mid M}
\Phi_{p^a}$.

With this definition, a binary cubic form $f$ corresponds
to a maximal cubic ring $R$ for~which $R\otimes\Q$
satisfies the local specifications of $\Sigma$ 
if and only if (a) 
$\Phi_M(f) = 1$, (b) $f \in U_p$ for each
prime $p \nmid M$, and (c) $f$ does not have splitting type $(1^3)$ at any prime
$p \nmid M$ for which $\Sigma_p=A_p'$.
For further details and explanation, see \cite[Section~4]{BST}, \cite[Section~5]{TT_L}, and
\cite[Section~6.2]{TT_rc}.

\subsection{Shintani zeta functions}\label{sec:shintani_zeta} Our proof will apply
the theory of {\itshape Shintani zeta functions}~\cite{shintani} 
associated to the space of integral binary cubic forms, 
defined by the Dirichlet series

\begin{equation}\label{eqn_shintani_def_1}
\xi^{\pm}(s) := \sum_{\substack{x \in \GL_2(\Z) \backslash V(\Z) \\ \pm \Disc(x) > 0}} \frac{1}{|\Stab(x)|} |\Disc(x)|^{-s},
\end{equation}
which (as we recall below) enjoys an analytic continuation and a functional equation.

By the work of Datskovsky and Wright \cite{DW1, DW2} (see also \cite{TT_L} and
F.\ Sato \cite{f_sato}), we may also consider the following generalization. Let
$m$ be any positive integer, let $\Phi_m: V(\Z/m\Z) \rightarrow \C$ be any function
such that $\Phi_m(\gamma x) = \Phi_m(x)$ for all $\gamma \in \GL_2(\Z/m\Z)$, and 
define 
\begin{equation}\label{eqn_shintani_def_2}
\xi^{\pm}(s, \Phi_m) = \sum_n a^{\pm}(\Phi_m, n) n^{-s} := \sum_{\substack{x \in \GL_2(\Z) \backslash V(\Z) \\ \pm \Disc(x) > 0}} \frac{1}{|\Stab(x)|} \Phi_m(x)
|\Disc(x)|^{-s}
\end{equation}
where we lift $\Phi_m$ to a function on $V(\Z)$.

The {\it dual} zeta functions ${\xi^{\ast,\pm}}(s, \Psi_m) = \sum_n a^{*,\pm}(\Psi_m, n) n^{-s}$ are defined, for each $\GL_2(\Z/m\Z)$-invariant
function $\Psi_m \ : \ V^*(\Z/m\Z) \rightarrow \C$, by a variant of \eqref{eqn_shintani_def_2}: the sum is now over all $\GL_2(\Z)$-orbits in the dual lattice $V^*(\Z)$. To define the discriminant of an element of $V^*(\Z)$, 
note that 
there is a $\GL_2(\Z)$-equivariant
embedding $\iota : V^*(\Z) \hookrightarrow V(\Z)$, whose image consists of those binary cubic forms 
whose middle two coefficients are divisible by $3$. Following Shintani, we 
identify $V^*(\Z)$ with its image in $V(\Z)$, and define the discriminant of 
an element of $V^*(\Z)$ via this embedding. 

We define $\widehat{\Phi_m} : V^*(\Z/m\Z) \rightarrow \C$ by the usual
Fourier transform formula
\begin{equation}\label{eta_hat}
\widehat{\Phi_m}(x) := \frac{1}{m^4} \sum_{y \in V(\Z/m\Z)} \Phi_m(y)
\exp\left(2 \pi i\cdot\frac{[x, y]}{m}\right),
\end{equation}
and lift $\widehat{\Phi_m}$ to a function on $V^*(\Z)$. We use $\iota$ to regard $\widehat{\Phi_m}$
as a function on $V(\Z)$, and write $\widehat{\Phi_m}(x) = 0$ for all $x \in V(\Z)$ not in the image of $\iota$.
With this convention, we have
$\xi^{\ast,\pm}(s,\widehat{\Phi_m})=\xi^\pm(s,\widehat{\Phi_m})$.

The functional equation has a particularly nice form when 
diagonalized, as observed in~\cite{DW2}.  To state the result, we write
$T=\left(\begin{smallmatrix}\sqrt3&1\\\sqrt3&-1\end{smallmatrix}\right)$
and
$\xi(s, \Phi_m)=\left(\begin{smallmatrix}\xi^+(s, \Phi_m)\\\xi^-(s, \Phi_m)\end{smallmatrix}\right)$, 
and introduce a gamma factor $\Delta(s):=\left(\begin{smallmatrix}\Delta^+(s)&0\\0&\Delta^-(s)\end{smallmatrix}\right)$, where
\begin{align*}
\Delta^+(s)&:=
\left(\frac{2^43^6}{\pi^4}\right)^{s/2}
		\Gamma\left(\frac s2\right)\Gamma\left(\frac s2+\frac12\right)
		\Gamma\left(\frac{s}{2}-\frac1{12}\right)
		\Gamma\left(\frac{s}{2}+\frac1{12}\right),\\
\Delta^-(s)&:=
\left(\frac{2^43^6}{\pi^4}\right)^{s/2}
		\Gamma\left(\frac s2\right)\Gamma\left(\frac s2+\frac12\right)
		\Gamma\left(\frac{s}{2}+\frac5{12}\right)
		\Gamma\left(\frac{s}{2}+\frac7{12}\right).
\end{align*}
\begin{theorem}\label{thm:fe}
The Shintani zeta functions $\xi^{\pm}(s, \Phi_m)$ converge absolutely for $\Re(s) > 1$, have analytic continuation to
all of $\C$, are holomorphic except for simple poles at $s = 1$ and $s = 5/6$, and satisfy the functional equation
\begin{equation}\label{eq:FE}
\Delta(1-s)\cdot T\cdot \xi(1-s, \Phi_m )= m^{4s} \left(\begin{smallmatrix}3&0\\0&-3\end{smallmatrix}\right) \cdot \Delta(s)\cdot T\cdot \xi(s, \widehat{\Phi_m}).
\end{equation}
The residues are given by
\begin{equation}\label{eqn_residue}
\Res_{s = 1} \xi^{\pm}(s, \Phi_m) = 
\alpha^{\pm} \mathscr{A}(\Phi_m) +
\beta \mathscr{B}(\Phi_m) \;\,\textrm{ and }\;\,
\Res_{s = 5/6} \ \xi^{\pm}(s, \Phi_m) = \gamma^{\pm} \mathscr{C}(\Phi_m),
\end{equation}
where
\begin{equation}\label{eqn_residues}
\alpha^+ = \frac{\pi^2}{72}, \ \ \alpha^- = \frac{\pi^2}{24}, \ \ \beta = \frac{\pi^2}{24} ,\ \ 
\gamma^+ = \frac{\pi^2 \zeta(1/3)}{9 \Gamma(2/3)^3},
\ \ \gamma^- = \sqrt{3} \gamma^+,
\end{equation}
and for functions of the form $\Phi_m = \otimes_{p^a \mid \mid m} \Phi_{p^a}$, we have
\begin{equation}
\mathscr{A}(\Phi_m) = \prod_{p^a \mid \mid m} \mathscr{A}(\Phi_{p^a}), \ \ 
\mathscr{B}(\Phi_m) = \prod_{p^a \mid \mid m} \mathscr{B}(\Phi_{p^a}), \textrm{ and }\,
\mathscr{C}(\Phi_m) = \prod_{p^a \mid \mid m} \mathscr{C}(\Phi_{p^a}),
\end{equation}
with $\mathscr{A}(\Phi_{p^a})$, $\mathscr{B}(\Phi_{p^a})$, and $\mathscr{C}(\Phi_{p^a})$ given explicitly as follows in the cases of interest below:
\vspace{.075in}\begin{enumerate}[{\rm (i)}]
\item 
For the characteristic function $\Phi_{p^2} : V(\Z/p^2\Z) \rightarrow \{ 0, 1 \}$
of those $x$ that correspond to non-maximal cubic rings over $\Z_p$:
\begin{equation}\label{eq:res_nmax}
\mathscr{A}(\Phi_{p^2}) = p^{-2} + p^{-3} - p^{-5}, \ \ 
\mathscr{B}(\Phi_{p^2}) = 2p^{-2} - p^{-4}, \ \ 
\mathscr{C}(\Phi_{p^2}) = p^{-5/3} + p^{-2} - p^{-11/3}.
\end{equation}
\vskip 4pt
\item 
For the characteristic function $\Phi_{p^2} : V(\Z/p^2\Z) \rightarrow \{ 0, 1 \}$ 
of those $x$ with $p^2 \mid \Disc(x)$:
\begin{equation}\label{eq:res_p2}
\mathscr{A}(\Phi_{p^2}) = 
\mathscr{B}(\Phi_{p^2}) = 2p^{-2} - p^{-4}, \ \ 
\mathscr{C}(\Phi_{p^2}) = p^{-5/3} + 2 p^{-2} - p^{-8/3} - p^{-3}.
\end{equation}
\vskip 4pt
\item 
For the characteristic function $\Phi_{p} : V(\Z/p\Z) \rightarrow \{ 0, 1 \}$ 
of those $x$ with $p \mid \Disc(x)$:
\begin{equation}\label{eq:res_p}
\mathscr{A}(\Phi_{p}) =
\mathscr{B}(\Phi_p) = 
 p^{-1} + p^{-2}- p^{-3}, \ \ 
\mathscr{C}(\Phi_{p}) = p^{-1} + p^{-4/3} - p^{-7/3}.
\end{equation}
\vskip 4pt
\item 
For the characteristic function $\Phi_{p^a} : V(\Z/p^a\Z) \rightarrow \{ 0, 1 \}$ 
corresponding to any of the five local splitting types $\calS_p$:
\begin{equation}\label{eq:res_other}
\mathscr{A}(\Phi_{p^a}) = C_p(\calS_p)  \left( 1 - \frac{1}{p^2} \right)  \left( 1 - \frac{1}{p^3} \right), \ \ 
\mathscr{C}(\Phi_{p^a})  = K_p(\calS_p) \left( 1 - \frac{1}{p^2} \right)  \left( 1 - \frac{1}{p^{5/3}} \right),
\end{equation}
where $C_p(\calS_p)$ and $K_p(\mathcal{S}_p)$ are as in Table $\ref{table:error_exponents}$. We further have 
that $\mathscr{B}(\Phi_{p^a}) = \delta \cdot \mathscr{A}(\Phi_{p^2})$, where $\delta = 3, 1, 0, 1, 0$ for each of these five splitting types, respectively.
\vspace{.1in}\item 
For the function $\Phi_{p} : V(\Z/p\Z) \rightarrow \{ 0, 1, 2, 3, p + 1 \}$,
where $\Phi_p(x)$ equals the number of roots of $x$ in $\P^1(\F_p)$:
\begin{equation}\label{eq:res_roots}
\mathscr{A}(\Phi_{p}) = 1 + p^{-1}, \ \ 
\mathscr{B}(\Phi_p) = 2, \ \
\mathscr{C}(\Phi_p) = 1 + p^{-1/3}.
\end{equation}
\end{enumerate}
\end{theorem}
\noindent
Parts (i)--(v) above are stated in this form in Propositions~8.15, 8.15, 8.14, 8.6 \& 8.13, and~8.12 of~\cite{TT_L}, respectively. (See also~Proposition 5.3 and
Theorem 6.2 of \cite{DW2}.)

We offer some remarks on normalization and how to deduce Theorem~\ref{thm:fe} from the literature. 
We have defined the zeta functions in terms of $\GL_2(\Z)$-orbits, rather than $\SL_2(\Z)$-orbits on $V(\Z)$ as used in \cite{shintani}
and \cite{TT_rc}. Each $\GL_2(\Z)$-orbit either splits into two $\SL_2(\Z)$-orbits of the same discriminant, or is a single $\SL_2(\Z)$-orbit, with $\GL_2(\Z)$-stabilizers of equal size or twice as large, respectively. Therefore the Shintani zeta function defined here is equal to half of that defined in \cite{shintani, TT_rc}. 
In \cite{TT_rc}, the $m^{4s}$ appeared in each term of the dual zeta function, instead of in the functional equation.
The functional equation of the zeta functions may be found in \cite[Theorem Q]{f_sato} or \cite[Theorem 4.3]{TT_L}, in the form
$\xi(1-s, \Phi_m )= m^{4s}M(s) \xi(s, \widehat{\Phi_m})$,
where $M(s)$ is defined by
\[
M(s):=
\frac{3^{6s-2}}{2\pi^{4s}}
\Gamma(s)^2\Gamma\left(s-\tfrac16\right)\Gamma\left(s+\tfrac16\right)
\begin{pmatrix}
\sin 2\pi s&\sin \pi s\\
3\sin \pi s&\sin2\pi s\\
\end{pmatrix}.
\]
This definition differs from \cite{TT_L} by a factor of $3^{3s}$, but coincides with that of Shintani, because of our choice
of normalization of the discriminant function on $V^*(\Z)$.
The diagonalized functional equation \eqref{eq:FE}
is obtained by plugging in the identity
\begin{equation}
\Delta(1-s)\cdot T\cdot M(s)
=\left(\begin{smallmatrix}3&0\\0&-3\end{smallmatrix}\right)
\cdot \Delta(s)\cdot T,
\end{equation}
which (after a change of variables) is due to Datskovsky and Wright \cite[Proposition~4.1]{DW2}.

\section{Error terms for Shintani zeta functions with local conditions via Landau's method}\label{sec:landau}

Let $\xi(s) := \sum_n a(n) n^{-s}$ be a ``zeta function'' with ``good analytic properties''; suppose, for example, that it converges absolutely
in the half-plane $\Re(s) > 1$, has a simple pole at $s = 1$ and no other poles on the line $\Re(s) = 1$,
enjoys a meromorphic continuation to $\C$, and satisfies a functional equation of the ``usual shape''. Then 
a classical method of Landau \cite{Landau1912,Landau1915}
establishes a power-saving estimate of the form
\[
\sum_{n < X} a(n) = X \cdot \Res_{s = 1} ( \xi(s) ) + o_\xi(X),
\]
where the $o$-function depends on $\xi$.
 In our case of interest, we have the following version:

\begin{theorem}\label{thm:landau}
Let $\xi^{\pm}(s, \Phi_m) := \sum_n a^\pm(\Phi_m, n) n^{-s}$ 
be the Shintani zeta function associated to a {\itshape nonnegative}
$\GL_2(\Z/m\Z)$-invariant function $\Phi_m \ : \ V(\Z / m \Z) \rightarrow \C$,
and let 
$\xi^{\pm}(s, \widehat{\Phi_m})
:= \sum_n a^{\pm}(\widehat{\Phi_m}, n) n^{-s}$
be the dual zeta function.
Let 
\begin{equation}
\delta_1 = \delta_1(\Phi_m) := \Res_{s = 1} \xi^{\pm}(s, \Phi_m)
\end{equation}
and
\begin{equation}\label{eqn:def_delta1}
\widehat{\delta_1} = \widehat{\delta_1}(\Phi_m) :=
m^4 \cdot \sup_{N} \frac{1}{N} \sum_{\alpha \in\{\pm\}} \sum_{n < N}
a^{\alpha}(|\widehat{\Phi_m}|, n).
\end{equation}
For a parameter $X > 0$, assume the following two technical conditions: 
\begin{equation}\label{eq:cond_lower_error}
\left|\Res_{s = 5/6} \xi^{\pm}(\Phi_m, s)\right|
\ll X^{1/6} \left|\Res_{s = 1} \xi^{\pm}(\Phi_m, s)\right|,
\end{equation}
\begin{equation}\label{eq:error_smaller}
\widehat{\delta_1} \ll \delta_1 X.
\end{equation}
Then 
\begin{equation}\label{eqn:landau}
N^\pm(X, \Phi_m) :=  \sum_{n < X} a^\pm(\Phi_m, n) = \sum_{\sigma \in \{ 1, \frac{5}{6} \}} \frac{X^{\sigma}}{\sigma} \cdot \Res_{s = \sigma} \xi^{\pm}(s, \Phi_m)
+ O\left( X^{3/5} \delta_1^{3/5} (\widehat{\delta_1})^{2/5} \right),
\end{equation}
where the implied constant does not depend on $\Phi_m$.
\end{theorem}

In \cite{LDTT}, Lowry-Duda and the second and third authors
gave a ``uniform version'' of Landau's method, closely following the 
exposition of Chandrasekharan and Narasimhan \cite{CN},
where the implied constants in \eqref{eqn:landau} depend only on 
the ``shape of the functional
equation'', and not on $\Phi_m$. (This was also done in a more {\itshape ad hoc} manner in \cite{TT_rc}.) The proof
is a bit easier if we incorporate the Datskovsky-Wright diagonalization (although this is not required; see \cite[Theorem 3]{SS}). 
Here we summarize the proof in 
 \cite{LDTT} while describing how to accommodate the diagonalization.
 
\begin{proof}
We apply Landau's method to the diagonalized zeta functions $\sqrt{3} \xi^+(s, \Phi_m)
\pm \xi^-(s, \Phi_m)$; we write $\xi(s, \Phi_m) = \sum_n a(\Phi_m, n) n^{-s}$ for either of them. Theorem \ref{thm:fe} gives (separate) functional
equations relating them to
$\xi(s,\widehat{\Phi_m}):=\sum a(\widehat{\Phi_m},n)n^{-s}=
\sqrt{3} {\xi}^+(s, \widehat{\Phi_m}) \pm 
{\xi}^-(s, \widehat{\Phi_m})$.

For each positive integer $k$, a variant of Perron's formula states that 
\begin{equation}\label{eqn:perron}
\frac{1}{\Gamma(k+1)}
\sum_{n < X} a(\Phi_m, n) (X - n)^{k} = \frac{1}{2 \pi i}
\int_{2 - i \infty}^{2 + i \infty} \xi(s, \Phi_m) \frac{X^{s+k}}{s(s + 1) \cdots (s + k)}
ds.
\end{equation}
Shifting \eqref{eqn:perron} to a line to the left of the critical strip,
the two main terms in \eqref{eqn:landau} come from the poles of $\xi(s)$. As $\xi(s)$ grows polynomially on this line 
as $|\Im(s)| \rightarrow \infty$, for
$k$ sufficiently large the integral in \eqref{eqn:perron} will converge absolutely.

We then use the functional equation, expand the dual zeta function $\xi(s, \widehat{\Phi_m})$ as an absolutely convergent Dirichlet series, and switch the order of summation
and integration.
Then~\eqref{eqn:perron} becomes
\[
\sum_{\sigma \in \{ 1, \frac{5}{6} \}}
\frac{X^{k+\sigma}}{\sigma(\sigma+1)\cdots(\sigma+k)} \cdot \Res_{s = \sigma} \xi(s, \Phi_m)
+ \frac{X^k}{k!}\cdot\xi(0, \Phi_m)
+m^{4k+4}
\sum_{n\geq1}
\frac{a(\widehat{\Phi_m},n)}{n^{k+1}}
I_k\left(\frac{nX}{m^4}\right),
\]
where
\[
I_k(t)
:=\pm 3 \cdot\frac{1}{2\pi i} \int_{c-i\infty}^{c+i\infty}
	\frac{\Delta^\pm(s)}{\Delta^\pm(1-s)}
	\cdot
	\frac{\Gamma(1-s)}{\Gamma(k+2-s)}	
	t^{k+1-s}ds
\]
for $c = \frac{9}{8}$, say.
A straightforward argument shows that $\xi(0, \Phi_m) \ll \widehat{\delta_1}(\Phi_m)$ (see \cite[Section~3]{LDTT}).
The integral $I_k(t)$ and its derivatives are treated
by approximating them by Bessel functions,
for which classical estimates are available.

We therefore obtain estimates for the smoothed partial sums on the left side of \eqref{eqn:perron},
for each of the two zeta functions  \smash{$\sqrt{3} \xi^+(s, \Phi_m)
\pm \xi^-(s, \Phi_m)$}. The last step in \cite{LDTT} is to consider~\eqref{eqn:perron}
with $X$ replaced by $X + iY$ for $0 \leq i \leq k$
for a parameter $Y$, and recover a formula
for $\sum_{n < X} a(\Phi_m, n)$ by finite differencing, within an error term depending on $Y$ (see \cite[Lemma~8]{LDTT}).

As described in \cite[(27)]{LDTT}, part of the error term may be eliminated if the coefficients $a(\Phi_m, n)$ are all nonnegative.
We therefore undo the diagonalization first, obtaining partial sum estimates $\Gamma(k+1)^{-1}\sum_{n < X} a^\pm(\Phi_m, n) (X - n)^{k}$
for the two zeta functions $\xi^\pm(s, \Phi_m)$ with the same error terms, up to an implied constant. These $a^\pm(\Phi_m, n)$
are nonnegative, since~$\Phi_m$ is assumed to be so, and we then proceed with the finite differencing.

The error term is simplified exactly as in the proof of \cite[Theorem 2]{LDTT}, with the factor of $m^4$ in 
\eqref{eqn:def_delta1} arising from our normalization of the dual zeta function above. The condition~\eqref{eq:cond_lower_error}
ensures that the error term arising from applying finite differencing to the residual term at $s = \frac56$ may be subsumed into the other error terms.

In conclusion, as long as $Y\ll X$, we have
\begin{equation}\label{eqn:landau-with-Y}
N^\pm(X, \Phi_m)
= \sum_{\sigma \in \{ 1, \frac{5}{6} \}} \frac{X^{\sigma}}{\sigma} \cdot \Res_{s = \sigma} \xi^{\pm}(s, \Phi_m)
+ O\left(\delta_1Y + \widehat{\delta_1}X^{3/2}Y^{-3/2} \right).
\end{equation}
The choice
$Y=X^{3/5}\delta_1^{-2/5}(\widehat{\delta_1})^{2/5}$
equalizes the two error terms,
with the condition $Y \ll X$ being equivalent to
\eqref{eq:error_smaller}, 
and we obtain the desired result.
\end{proof}

We also have the following ``average version'':

\begin{theorem}\label{thm:landau_average}
Keeping the notation of Theorem $\ref{thm:landau}$, consider an arbitrary finite set of pairs
$(\Phi_{m_i}, X_i)_{i \in \mathcal{I}}$, where each 
$\Phi_{m_i}$ is a function satisfying the hypotheses of Theorem $\ref{thm:landau}$ other than \eqref{eq:error_smaller}, and each $X_i$ is a positive real number with $X_i \asymp X$
for some fixed $X$. Then
\begin{equation}\label{eq:landau_2}
\sum_{i \in \mathcal{I}} \Biggl| N^\pm(X_i, \Phi_{m_i}) - \sum_{\sigma \in \{ 1, \frac{5}{6} \}} \frac{X_i^{\sigma}}{\sigma} \cdot \Res_{s = \sigma} \xi^{\pm}(s, \Phi_{m_i}) \Biggr|
\ll 
X^{3/5}
\left(
\sum_{i \in \mathcal{I}}
 \delta_1(\Phi_{m_i}) \right)^{3/5}
 \!\left(
 \sum_{i \in \mathcal{I}}
 \widehat{\delta_1}(\Phi_{m_i}) \right)^{2/5}
\end{equation}
provided that, in place of \eqref{eq:error_smaller}, we have 
\begin{equation}\label{eq:error_smaller_2}
\sum_{i \in \mathcal{I}} \widehat{\delta_1}(\Phi_{m_i}) \ll X \sum_{i \in \mathcal{I}} \delta_1(\Phi_{m_i}).
\end{equation}

\end{theorem}

\begin{proof}
The proof is identical to that of Theorem~\ref{thm:landau}, except that in \eqref{eqn:landau-with-Y} we 
make the choice 
\[
Y=
X^{3/5}
\left( \sum_{i \in \mathcal{I}} \delta_1(\Phi_{m_i}) \right)^{-2/5} 
\left( \sum_{i \in \mathcal{I}} \widehat{\delta_1}(\Phi_{m_i}) \right)^{2/5} 
\] for all $i \in \mathcal{I}$
simultaneously.
\end{proof}

\begin{remark}
The condition \eqref{eq:cond_lower_error} is immediate from 
the residue formulas in Theorem \ref{thm:fe}, for $m \ll X^{1 - \epsilon}$ and 
any $\Phi_m$ whose residues are described there. 
Outside of Remark \ref{remark:other_split}, this covers all Shintani
zeta functions introduced in this paper.
\end{remark}

\section{Additional preliminaries on reducible rings and uniformity estimates}\label{sec:prelim}

\subsection{Reducible rings and their weighting}

By the Levi--Delone--Faddeev correspondence, the Shintani zeta functions may also be written as
\begin{equation}\label{eqn:shintani_df}
\xi^{\pm}(s) := \sum_{\pm \Disc(R) > 0} \frac{1}{|\Aut(R)|} |\Disc(R)|^{-s},
\end{equation}
where the sum is over all isomorphism classes of cubic rings. The main work (here, and also in previous papers on Davenport--Heilbronn)
is to sieve for maximality. After doing so, we will obtain a count for the quantity
\begin{equation}
N^{\pm}_{\leq 3}(X) := \sum_{0 < \pm \Disc(F) < X} |\Aut(F)|^{-1},
\end{equation}
where the sum is over \'etale algebras $F$ of degree $3$, i.e.,
direct products of number fields whose degrees sum to $3$.
Note that these are in bijection with
number fields of degrees at most $3$, and that the automorphism group of such an algebra is naturally isomorphic to that of its maximal order. 
For any collection of cubic local specifications $\Sigma$,
we define $N_{\leq 3}^{\pm}(X, \Sigma)$ 
analogously. In Proposition \ref{prop:red1}, we relate these quantities to the quantities $N^{\pm}_3(X)$ and $N^{\pm}_3(X, \Sigma)$ introduced in Theorems \ref{thm_rc}
and \ref{thm_lc}.

As in \cite{cohn}, the following lemma 
follows from 
Hasse's characterization \cite{hasse} of 
cyclic cubic field discriminants:
\begin{lemma}\label{cohn} Let
$t$ be
a squarefree integer.
Then the number of cyclic cubic fields $F$ with $|\Disc(F)| < X$ 
and $t \mid \Disc(F)$
is $O(X^{1/2}/ t^{1 - \epsilon})$.
\end{lemma}

We also require an estimate for the contribution
of quadratic fields, optionally with local conditions.
Suppose, as before, that 
$\Sigma = (\Sigma_p)_p$ is a collection of cubic local specifications
that are ordinary away from an integer $M$, the conductor of $\Sigma$.
We say that a {\itshape quadratic} field $F$
satisfies the cubic local specifications of $\Sigma_p$ if the cubic \'etale
algebra $F \times \Q$ does.

Except for $\Q^3$, reducible cubic \'etale algebras take 
the form $F\times\Q$ where $F$ is a quadratic field. For the reducible algebra $F\times\Q$,
the
only nonempty conditions that $\Sigma$ imposes on $F$ are those
at primes dividing $M$. We write $U$, $s$, and $r$ for the product
of those prime
factors~$p$ of $M$ for which, respectively: $\Sigma_p$ contains at least one unramified algebra;
$\Sigma_p$ consists of a single algebra $F\times \Q_p$ where $F$ is ramified; $\Sigma_p$ consists of all algebras $F\times \Q_p$ where $F$ is~ramified.

The following proposition counting quadratic fields satisfying a collection of local specifications is a slight generalization 
of 
\cite[Lemma 6.1]{TT_rc} and
\cite[Proposition 8.1]{EPW}, and with an improved
error term. We will give a complete and self-contained
proof in Section~\ref{appendix:count_quadratic}.

\begin{proposition}\label{lem:quad_count}
Write $N_2^{\pm}(X, \Sigma)$ for the number of isomorphism classes of quadratic fields $F$ with $0 < \pm \Disc(F) < X$ 
that satisfy a collection of local specifications $\Sigma$ with $M$, $U$, $s$, and $r$ as above.
Then
\begin{equation}
N_2^{\pm}(X, \Sigma) = \prod_p 
c_p(\Sigma_p) \cdot \frac{3}{\pi^2} X
+O\left(X^{1/2} \frac{U^{1/4}}{s^{1/4}r^{1/2}}(Usr)^\epsilon\right),
\end{equation}
where
\begin{equation}
c_p(\Sigma_p) = \frac{\text{\footnotesize $\displaystyle\sum_{F\,:\,F\times\Q_p\in\Sigma_p}$}
\!\textstyle\frac{1}{\Disc_p(F)}\frac{1}{|\Aut(F)|}}{\text{\footnotesize$\displaystyle\sum_{\scriptstyle F\,:\,F\times\Q_p\in A_p}$} \!\textstyle\frac{1}{\Disc_p(F)}\frac{1}{|\Aut(F)|} }.
\end{equation}
\end{proposition}

In Proposition~\ref{lem:quad_count}, we may explicitly evaluate $c_p(\Sigma_p)$ for specific $\Sigma_p$ as follows:
\begin{itemize}
\item
If $p\nmid M$, then $c_p(\Sigma_p)=1$ and so it may be omitted from the product. \smallskip
\item If $\Sigma_p$ consists of a single unramified and reducible algebra, then $c_p(\Sigma_p) := \frac{1}{2} (1 + p^{-1})^{-1}$.\smallskip
\item If $p\neq2$, then there are 
two ramified quadratic extensions $F$ of $\Q_p$. If $\Sigma_p$ consists of $F\times \Q_p$ for one of these $F$, then $c_p(\Sigma_p) = \frac{1}{2} (p + 1)^{-1}$. \smallskip
\item  If $\Sigma_p$ consists of all algebras $F\times\Q_p$ where $F$ is ramified,  then $c_p(\Sigma_p) := (p + 1)^{-1}$.  \smallskip
\item 
There are six ramified quadratic extensions
of $\Q_2$; if $\Sigma_2$ consists one of these six extensions, then the associated values of $c_2(\Sigma_2)$
are given in Section
$\ref{appendix:count_quadratic}$. \smallskip
\item
If $\Sigma_p$ contains more than 
one algebra, then
$c_p(\Sigma_p)$ is the sum of
the values for the individual algebras.
\end{itemize}

The following formula relates
$N_3^\pm(X, \Sigma)$ to $N_{\leq 3}^\pm(X, \Sigma)$.
\begin{proposition}\label{prop:red1}
We have
\begin{equation}
N_3^{\pm}(X, \Sigma) = 
N^{\pm}_{\leq 3}(X, \Sigma) - 
\frac{1}{2} N_2^{\pm}(X, \Sigma) + O(X^{1/2}/t^{1-\epsilon}),
\end{equation}
where 
$t$ is the product of those
primes $p$ such that $\Sigma_p$ consists of only totally ramified cubic extensions of $\Q_p$.
\end{proposition}

\begin{proof}
This is immediate from Lemma~\ref{cohn} and the formula
\begin{equation}
|\Aut(F)| = 
\begin{cases}
	1	& \text{if $F$ is a non-Galois cubic field},\\
	3	& \text{if $F$ is a Galois cubic field},\\
	2	& \text{if $F = F_2 \times \Q$ with $F_2$ a quadratic field},\\
	6	& \text{if $F = \Q \times \Q \times \Q$.}\\
\end{cases}\vspace{-.1in}
\end{equation}
\end{proof}
\begin{remark}\label{rmk:quad_error_subsumed}
Note that the error term in Proposition \ref{lem:quad_count}
is always bounded above by that of our main
theorems (Theorem \ref{thm_lc} and \ref{thm_lc_averaged}).
\end{remark}

\subsection{Uniformity/tail estimates}\label{sec:tail} Our basic uniformity/tail estimate is a variation of \cite[Proposition~1, p.~410]{DH}
and \cite[Lemma 3.4]{BBP}, which bounds the number of cubic rings having absolute discriminant less than $X$ and divisible by $q^2$, where $q$ is squarefree; 
the exact statement can be found as \cite[Lemma 3.4]{BBP} (see also \cite[Lemma~3.4]{TT_rc}).

\begin{proposition}\label{prop:tail_no_eps}
For each squarefree integer $q$,
the number of cubic rings $R$ with $q^2 \mid \Disc(R)$ and $0 < \pm \Disc(R) < X$ is $\ll 6^{\omega(q)} X/q^2$.
\end{proposition}
\noindent Here (and elsewhere), $\omega(q)$ is the number of prime divisors of $q$.

For nonmaximal rings, we have the following stronger estimate.
The proof is identical to \cite[Lemma 2.7]{BBP},
using Lemma \ref{lemma_overrings} to include the reducible rings.
\begin{proposition}\label{prop:tail_nonmax}
For each squarefree integer $q$,
the number of cubic rings $R$ that are nonmaximal at each prime divisor of $q$ and satisfy $0 < \pm \Disc(R) < X$ is $\ll 3^{\omega(q)} X/q^2$. 
\end{proposition}

Finally, we require the following uniformity estimate on the number of cubic rings having discriminant divisible by a given cubefree integer $q$:
\begin{proposition}\label{prop:uniformity}
For each cubefree integer $q < X^{1/4-\epsilon'}$,
the number of cubic rings $R$ with $0 < \pm \Disc(R) < X$ for which $q \mid \Disc(R)$ 
is $\ll X/q^{1-\epsilon}$
where $\epsilon,\epsilon'>0$ are arbitrary.
\end{proposition}

\begin{proof}
Let $\Phi_q : V(\Z/q\Z) \rightarrow \{ 0, 1 \}$ be the characteristic function of those forms whose discriminants are divisible
by $q$. Then, we have
\begin{equation}
\sum_{n < X} a^{\pm}(\Phi_q, n) \leq e \cdot 
\sum_{n} a^{\pm}(\Phi_q, n) \exp(-n/X)
= 
\frac{e}{2 \pi i} \int_{2 - i \infty}^{2 + i \infty} \xi^\pm(s,\Phi_q) 
X^s \Gamma(s) ds.
\end{equation}
Shift the contour to $\Re(s) = - c$ for some small $c > 0$, obtaining a `main term'
of $O(X/q^{1-\epsilon})$ from~the pole of $\xi^{\pm}(s,\Phi_q)$ at $s = 1$.
We have the trivial bound
\smash{$|\widehat\Phi_q(x)|\ll q^{-1+\epsilon''}$},
since the support of $\Phi_q$ in $V(\Z/q\Z)$ has cardinality
$\ll q^{3+\epsilon''}$.
Thus the functional
equation establishes a bound of
$O(q^{3 + 4c+\epsilon''})$ on the line $\Re(s) = -c$
of the integrand,
and the factor of $\Gamma(s)$ guarantees
the absolute convergence of the integral. The contribution of the other residues may be either estimated directly,
or bounded using the Phragmen-Lindel\"of principle, and this completes the proof. 
\end{proof}

\begin{remark}
A similar result also appeared as \cite[Lemma 4.4]{TT_rc}. We will actually only need this result for $q < X^{\delta}$ for {\itshape any} arbitrarily small $\delta > 0$.
This result is used in the penultimate paragraph of Section~5 (and analogously in Section~6) where, after applying the functional equation to the relevant Shintani zeta functions, we must bound a `dual sum' on $V^*(\Z)$
that shares some characteristics with our original counting problem. A result of this form could also be
immediately deduced from the main results of this paper or from~\cite{TT_rc}. The above proof is independent of our other results and thus allows us to avoid any circular reasoning.
\end{remark}

\section{Direct proof}\label{sec:direct_proof}
We first prove Theorems \ref{thm_rc} and \ref{thm_rc_torsion} directly using Landau's method as presented in Theorems
\ref{thm:landau} and \ref{thm:landau_average}.
We obtain error	terms of $O(X^{2/3+\epsilon})$ in both
	results. (For the improvement to 
	$O(X^{2/3} (\log X)^{2.09})$ in Theorem
	\ref{thm_rc}, see the alternative
	proof using discriminant reduction in Section
	\ref{sec:disc_reduce_1}.)

We fix a collection of local specifications $\Sigma=(\Sigma_p)_p$,
ordinary away from its conductor~$M$, such that
the conditions (mod $M$) correspond to 
a $\GL_2(\Z/M\Z)$-invariant
function $\Phi_M : V(\Z / M \Z) \rightarrow \C$.
(For Theorems \ref{thm_rc}
and \ref{thm_rc_torsion} we take
$M = 1$; the same setup will also be used
in Section~\ref{sec:lc},
when proving the more
general
Theorem \ref{thm_lc_averaged}.)
As before, we may assume that $\Phi_M$ is
supported only on those $x \in V(\Z / M \Z)$ satisfying the Davenport--Heilbronn maximality condition of Proposition 
\ref{def_up} for each prime $p \mid M$.

To carry out the sieve, for each prime $p \nmid M$ 
we write $\Psi_{p^2} : V(\Z/p^2\Z) \rightarrow \{ 0, 1 \}$
for the characteristic function of
one of the following two sets, in accordance with whether $\Sigma_p = A_p$ or $\Sigma_p = A'_p$:
\begin{itemize}
\item
Those $x \in V(\Z/p^2\Z)$ that are nonmaximal at $p$ 
in the sense of Proposition \ref{def_up}. \smallskip
\item
Those $x \in V(\Z/p^2\Z)$ that are nonmaximal at $p$ or have a triple root $(\textmod \ p)$.
When $p > 2$, this is equivalent to requiring that $p^2 \mid \Disc(x)$.
\end{itemize}
For each squarefree $q$ coprime to $M$,
we view $\Psi_{q^2} \!:=\!
\otimes_{p \mid q} \Psi_{p^2}$ as a function
$V(\Z/q^2\Z) \rightarrow \{ 0, 1 \}$.

The quantity $N^{\pm}_3(X, \Sigma)$ of 
interest is related to $N^{\pm}_{\leq 3}(X, \Sigma)$ as discussed in Proposition~\ref{prop:red1},
so that it suffices to estimate the latter.
The Levi--Delone--Faddeev correspondence
(Theorem~\ref{thm:df}), the Davenport--Heilbronn
correspondence (Proposition \ref{def_up}),
and inclusion-exclusion give
\begin{equation}
N^{\pm}_{\leq 3}(X, \Sigma)  = 
\sum_q \mu(q) N^{\pm}(X, \Phi_M \Psi_{q^2}).
\end{equation}
We now split the sum into two parts
in accordance with whether $q\leq Q$
or $q>Q$, and apply
Landau's method (Theorem \ref{thm:landau_average})
for the former.
We obtain
\begin{equation}\label{eq:main_term}
N^{\pm}_{\leq 3}(X, \Sigma)
= \sum_{\sigma \in \{ 1, \frac{5}{6} \}} \frac{X^\sigma}{\sigma}
\sum_{\substack{q =1 \\ (q, M) = 1}}^{\infty} \mu(q) \cdot \Res_{s = \sigma} \xi^{\pm}(s, \Phi_M \Psi_{q^2})
+ O\left( E_1 + E_2 + E_3 \right),
\end{equation}
with
\begin{align*}
E_1 := & \ \sum_{\sigma \in \{ 1, \frac56 \}} X^{\sigma}  \sum_{\substack{q > Q}} \left|\Res_{s = \sigma} \xi^{\pm}(s, \Phi_M \Psi_{q^2})\right|,
\\[-.065in]
E_2 := & 
X^{\frac35}
\sum_{Q_1} 
\Biggl( \sum_{q \in [Q_1, 2Q_1]} \delta_1(\Phi_M \Psi_{q^2}) \Biggr)^{\frac35} 
\Biggl( \sum_{q \in [Q_1, 2Q_1]} \widehat{\delta_1}(\Phi_M \Psi_{q^2}) \Biggr)^{\frac25},
\\[.05in]
E_3 := &
\sum_{q > Q} N^{\pm}(X, \Phi_M \Psi_{q^2}),
\end{align*}
provided that \eqref{eq:error_smaller_2} is satisfied for our choice of the parameter $Q$.
Here $Q_1$ ranges over all integer powers of $2$ less than $Q$.
Throughout, all summations over $q$ are over {\itshape squarefree} integers
coprime to $M$.

In this section, we specialize to the 
case $M = 1$, i.e., when $\Sigma_p = A_p$ or $\Sigma_p =  A'_p$ for every prime $p$.
(Making the same choice for all primes $p$ leads to Theorem \ref{thm_rc} or Theorem \ref{thm_rc_torsion} respectively.) In this case, we have
$\Phi_M \Psi_{q^2} = \Psi_{q^2}$, and
our residue formulas \eqref{eq:res_nmax} and \eqref{eq:res_p2} then imply that 
\[
E_1 \ll X \sum_{q > Q} q^{-2 + \epsilon} + X^{5/6}  \sum_{q > Q} q^{-5/3 + \epsilon}
\ll \frac{X}{Q^{1 - \epsilon}}
\]
for $Q < X^{1/2}$, 
and that \eqref{eq:cond_lower_error} is satisfied in the same range.
By the tail estimate of Proposition~\ref{prop:tail_no_eps}, we also have
\[
E_3 \ll X \sum_{q > Q} 6^{\omega(q)} q^{-2} \ll \frac{X}{Q^{1 - \epsilon}}.
\]

To bound $E_2$, note that 
$\delta_1( \Psi_{q^2}) \ll q^{-2 + \epsilon}$ for each $q$ by the residue formulas already quoted, 
and we claim the following average bound on $\widehat{\delta_1}(\Psi_{q^2})$:
\begin{proposition}\label{prop:bound_delta_2}
For either definition of $\Psi_{q^2}$, we have
\begin{equation}\label{eqn:bound_delta_2}
\sum_{q \in [Q, 2Q]} \widehat{\delta_1}(\Psi_{q^2}) \ll Q^{2 + \epsilon}.
\end{equation}
\end{proposition}
\noindent Granting this for now, we have
\[
E_1 + E_2 + E_3 \ll \frac{X}{Q^{1 - \epsilon}} + X^{3/5} 
\sum_{Q_1} Q_1^{- \frac35 + \epsilon} \cdot Q_1^{\frac45 + \epsilon}
\ll
\frac{X}{Q^{1 - \epsilon}} + X^{3/5} Q^{1/5 + \epsilon}.  
\]
Choosing \smash{$Q = X^{1/3 - \epsilon}$}, which is acceptable in \eqref{eq:error_smaller_2}, we obtain error terms of $O(X^{2/3 + \epsilon})$ in Theorem
\ref{thm_rc} and in
\eqref{eqn_torsion2}, which is equivalent to Theorem \ref{thm_rc_torsion}. The main terms in counting $N_{\leq 3}^\pm(X)$
are given by the infinite sums in \eqref{eq:main_term}, namely
\begin{equation}\label{eq:sums}
C^\pm \cdot \frac{\pi^2}{72} \cdot X \prod_p (1 - \mathscr{A}(\Psi_{p^2})) +
\frac{\pi^2}{24} \cdot X \prod_p (1 - \mathscr{B}(\Psi_{p^2})) +
K^\pm \cdot \frac{\pi^2 \zeta(1/3)}{9 \Gamma(2/3)^3} \cdot \frac{6}{5}
X^{5/6} \cdot \prod_p  (1 - \mathscr{C}(\Psi_{p^2})).
\end{equation}
The products are evaluated using the formulas in \eqref{eq:res_nmax} and \eqref{eq:res_p2} respectively, and then
an application of Propositions  \ref{lem:quad_count} and \ref{prop:red1} finishes the proof.

It remains only to prove Proposition \ref{prop:bound_delta_2}. 
Expanding the definition \eqref{eqn:def_delta1} of $\widehat{\delta_1}$, the bound
to be proved is
\begin{equation}\label{eqn:bound_delta_2a}
\sum_{\alpha\in\{\pm\}}\sum_{q \in [Q, 2Q]}  
\sum_{n<N} a^{\alpha}(|\widehat{\Psi_{q^2}}|, n)
\ll N Q^{-6 + \epsilon}.
\end{equation}
By the definition of the coefficients, this is equivalent to
showing that
\begin{equation}\label{eqn:bound_delta_2b}
\sum_{q \in [Q, 2Q]}
\sum_{\substack{x \in \GL_2(\Z) \backslash V(\Z) \\0\neq|\Disc(x)|<N}}
|\widehat{\Psi_{q^2}}(x)|
\ll N Q^{-6 + \epsilon}.
\end{equation}
Our analysis will bound the sum over $q\in [Q,2Q]$ as a whole, and avoid the need for sharp bounds on the inner sum for each $q$ individually. 

The first step is an evaluation of the Fourier transforms \smash{$|\widehat{\Psi_{q^2}}(x)|$}; the following
combination of \cite[Lemmas 3.3 and 4.3]{TT_rc} summarizes\footnote{We point out a minor mistake
in Lemmas 3.3 and 4.3 of \cite{TT_rc}: exact formulas are claimed for the values of $\Psi_{q^2}$ for those $x$
which are nonmaximal at $p$, have content coprime to $p$, and for which $p^4 \mid \Disc(x)$. These formulas
hold as equalities only for some such $x$, but as upper bounds they hold for all $x$.} the results we need:

\begin{proposition}\label{lem:ft_eval} For either definition of $\Psi_{q^2}$, the function $\widehat{\Psi_{q^2}}(x)$ is multiplicative
in $q$, and satisfies the following bounds (where $R$ is the cubic ring corresponding to $x$):
\begin{itemize}
\item $($Content $p^2)$\, $|\widehat{\Psi_{p^2}}(x)| = O(p^{-2})$ if $p^2$ divides the content of $R$.
\item $($Content $p)$\, $|\widehat{\Psi_{p^2}}(x)| = O(p^{-3})$ if $p$ divides the content of $R$ but $p^2$
does not.
\item $($Divisible by $p^4)$\, $|\widehat{\Psi_{p^2}}(x)| = O(p^{-3})$ if $R$ is nonmaximal at $p$,
$p$ does not divide the content of $R$, and $p^4 \mid \Disc(R)$.
\item $($Divisible by $p^3)$\, $|\widehat{\Psi_{p^2}}(x)| = O(p^{-4})$ if $R$ is nonmaximal at $p$
and $p^3 \mid \mid \Disc(R)$.
\item $($Divisible by $p^2)$\, $|\widehat{\Psi_{p^2}}(x)| = O(p^{-5})$ if $p^2 \mid \mid \Disc(R)$.
\item Otherwise, and in particular if $p^2 \nmid \Disc(R)$, we have $\widehat{\Psi_{p^2}}(x) = 0$.
\end{itemize}
\end{proposition}

\noindent
For each squarefree $q \in [Q, 2Q]$, we consider the contribution to \eqref{eqn:bound_delta_2b} from 
every factorization
\begin{equation}\label{eq:factor1}
q = c_2 c_1 d_4 d_3 d_2
\end{equation}
and from those $x$ such that $x \!\pmod {p^2}$ satisfies the first five conditions of Proposition \ref{lem:ft_eval} 
for $p$ dividing $c_2$, $c_1$, $d_4$, $d_3$, and $d_2$, respectively. 
We begin by replacing each form $x$ with $x/c_2^2 c_1$, using the natural bijection 
between forms $x$ with $|\Disc(x)| < N$ and content divisible by $c_2^2 c_1$, 
and forms $x$ with $|\Disc(x)| < Nc_2^{-8} c_1^{-4}$.

By Lemma \ref{lemma_overrings} (i),
each ring $R$ with $d_4 d_3>1$ is contained
in an overring $R'$ of index $d_4 d_3$. The content of $R'$ may be divisible by prime factors
of $d_4$; we write $d_4 = d_{4c} d_{4n}$, where $d_{4c}$ is the gcd of $d_4$ and the content of $R'$, and refine the factorization
of \eqref{eq:factor1} to 
\begin{equation}\label{eq:factor2}
q = c_2 c_1 d_{4c} d_{4n} d_3 d_2.
\end{equation}
We count our rings $R$ by counting these overrings $R'$ with multiplicity 
given by the number of $R$ thus contained in any such $R'$, which is 
$\ll q^{\epsilon} d_{4c}$ by Lemma \ref{lemma_overrings} (ii).
We then replace 
the form $x'$ corresponding to each such $R'$
by $x'/d_{4c}$.

The total contribution to \eqref{eqn:bound_delta_2b} is therefore bounded above by
\begin{equation}\label{eqn:total_contribution}
Q^{\epsilon} \sum_{c_2, c_1, d_{4c}, d_{4n}, d_3, d_2}
\frac{1}{c_2^2 c_1^3 d_{4c}^2 d_{4n}^3 d_3^4 d_2^5}
\sum_{\substack{|\Disc(x')|  < \frac{N}{c_2^8 c_1^4 d_{4c}^6 d_{4n}^2 d_3^2} \\
d_{4n}^2 d_3 d_2^2 \mid \Disc(x')}} 1,
\end{equation}
where the outer sum is over all choices of the six variables whose product \eqref{eq:factor2} is in $[Q, 2Q]$.

When $N$ is not too large---when $N \leq Q^{100}$, say--- we first sum the variable $d_3$ over dyadic intervals $[D_3, 2D_3]$. 
For each fixed $x'$ in the inner sum, 
there are at most $O(Q^{\epsilon})$  such $d_3$ with $d_3 \mid \Disc(x')$, 
and so 
the sum in \eqref{eqn:total_contribution} is bounded above by
\begin{equation}\label{eqn:total_contribution2}
Q^{\epsilon} \sum_{D_3}
\sum_{c_2, c_1, d_{4c}, d_{4n}, d_2}
\frac{1}{c_2^2 c_1^3 d_{4c}^2 d_{4n}^3 D_3^4 d_2^5}
\sum_{\substack{|\Disc(x')| < \frac{N}{c_2^8 c_1^4 d_{4c}^6 d_{4n}^2 D_3^2} \\
d_{4n}^2 d_2^2 \mid \Disc(x')}} 1.
\end{equation}
By Proposition \ref{prop:tail_no_eps}, the inner sum is 
$\ll Q^{\epsilon} \frac{N}{c_2^8 c_1^4 d_{4c}^6 d_{4n}^2 D_3^2} \cdot \frac{1}{d_{4n}^2 d_2^2}$,
so that the above simplifies~to
\begin{equation}\label{eqn:messy2}
Q^{\epsilon}
\sum_{D_3} 
\sum_{c_2, c_1, d_{4c}, d_{4n}, d_2}
\frac{N}{c_2^{10} c_1^7 d_{4c}^8 d_{4n}^7 D_3^6 d_2^7}.
\end{equation}
The double sum is over choices of the six variables whose product is $> Q$, 
so that the sum 
is $\ll N Q^{-6 + \epsilon}$.

For $N > Q^{100}$, Proposition \ref{prop:uniformity} more than suffices to bound the inner sum in \eqref{eqn:total_contribution},
again yielding a bound $\ll N Q^{-6 + \epsilon}$. 

We have thus proven 
\eqref{eqn:bound_delta_2b} in both cases, and therefore
Proposition \ref{prop:bound_delta_2}. This yields Theorem~\ref{thm_rc}
(with $(\log X)^{2.09}$ replaced with $X^\epsilon$)
and \ref{thm_rc_torsion}.

\section{Proof of Theorem \ref{thm_lc_averaged}: Local Conditions}\label{sec:lc}

Let $U$, $r$, and $t$ be squarefree and coprime integers, and let $\Sigma^{r, t}$ be a collection of local specifications satisfying the hypotheses
described immediately before Theorem \ref{thm_lc_averaged}.
For simplicity, we 
assume for each $p \mid U$ (and we know a priori for each $p \mid rt$) that $\Sigma_p$ consists of all extensions having 
one of the five cubic splitting types; by summation, this will imply
the result when $\Sigma_p$ is a union of splitting types.
In Remark \ref{remark:other_split}, we describe
how to amend the proof to accommodate $\Sigma_p$ that are not a union of splitting types.

We always assume that $r$ and $t$ are coprime to $6$, incorporating any $2$- and $3$-adic conditions into the $U$ component instead.
Conversely, for each $p > 3$ dividing $U$, we assume that $\Sigma_p$ corresponds to
one of the three {\itshape unramified} splitting types; ramified splitting types may be incorporated using the variables $r$ and $t$ instead.

As described in Section \ref{subsec:lc}, there is an integer $M$, called the {\itshape conductor} of $\Sigma^{r, t}$ and 
divisible by precisely the primes dividing $U rt$,
such that the partial collection $(\Sigma_p)_{p \mid U r t}$ corresponds to a function
$\Phi_M : V(\Z/M\Z) \rightarrow \{0, 1 \}$ of the form $\otimes_{p^a \mid \mid M} \Phi_{p^a}$, where $\Phi_{p^a}$ is any of the following functions:
\begin{itemize}
\item $\Theta_{(111), p}, \Theta_{(21), p}, \Theta_{(3), p} : V(\Z/p\Z) \rightarrow \{0, 1\}$, 
the characteristic functions of the three unramified splitting types (totally split, partially split, or inert), respectively. These functions are automatically supported only on forms that are maximal at $p$. \smallskip
\item $\Theta_{(1^2 1), p^2} : V(\Z/p^2\Z) \rightarrow \{0, 1\}$, the characteristic function of the partially ramified splitting type, corresponding to binary cubic forms that are maximal at $p$ and have exactly two roots in~$\P^1(\F_p)$, one of which is a double root.
\smallskip
\item $\Theta_{(1^3), p^2} : V(\Z/p^2\Z) \rightarrow \{0, 1\}$, the characteristic function of the totally ramified  splitting type, corresponding to binary 
cubic forms that are maximal at $p$ and have a triple root in $\P^1(\F_p)$. 
\end{itemize}

To accommodate our analysis, we also introduce the following
additional functions:
\begin{itemize}
\item $\Theta_{\textdivblahblah, p} : V(\Z/p\Z) \rightarrow \{0, 1\}$, the characteristic function of binary cubic forms $x \!\pmod{p}$ for which 
$p \mid \Disc(x)$, but without any maximality condition. \smallskip
\item $\Theta_{\textdivblahblah^2, p^2} : V(\Z/p^2\Z) \rightarrow \{0, 1\}$, the characteristic function of forms $x$ for which 
$p^2 \mid \Disc(x)$, but without any maximality condition. \smallskip
\item $\Theta_{\textnmax, p^2} : V(\Z/p^2\Z) \rightarrow \{0, 1\}$, the characteristic function of those forms that are nonmaximal at $p$.
\end{itemize}
The last two functions were introduced with the notation $\Psi_{q^2}$ in the previous section, and 
Fourier transform bounds for these functions were obtained
in Proposition
\ref{lem:ft_eval}.
For $p\neq 2, 3$, we have the relation
\begin{equation}\label{eq:Phip_relations}
\Theta_{(1^2 1), {p^2}} = \Theta_{\textdivblahblah, p} - \Theta_{\textdivblahblah^2, {p^2}}
\end{equation}
when one views $\Theta_{\textdivblahblah, p}$ as a function on $V(\Z/p^2\Z)$. 
We also have
$\Theta_{(1^3), {p^2}} = \Theta_{\textdivblahblah^2, {p^2}} - \Theta_{\textnmax, {p^2}}$, implying that $\Theta_{(1^3), {p^2}}$ satisfies the same Fourier transform bounds in Proposition \ref{lem:ft_eval}.

We use the relation \eqref{eq:Phip_relations} to write
\begin{equation}\label{eq:Phip2}
N_{\leq 3}^\pm(X, \Sigma^{r, t})
= 
\sum_{d \mid r} \mu(d) \sum_{\substack{q =1 \\ (q, M) = 1}}^{\infty} \mu(q)
N^{\pm}(X, 
\Phi_U \Theta_{\textdivblahblah, \frac{r}{d}} \Theta_{\textdivblahblah^2, d^2} \Theta_{(1^3), t^2} \Psi_{q^2}),
\end{equation}
where
$\Phi_U = \otimes_{p^a \mid \mid U} \Phi_{p^a}$, with 
$\Phi_{p^a}$ denoting any of
$\Theta_{(111), p}, \Theta_{(21), p},$ or $\Theta_{(3), p}$ for $p \nmid 6$, and 
$\Psi_{q^2} = \prod_{p \mid q} \Psi_{p^2}$, with
$\Psi_{p^2}$ denoting one of the two functions according as $\Sigma_p = A_p$
or $\Sigma_p = A'_p$ as in Section \ref{sec:direct_proof}.
We let 
\begin{equation}\label{eq:exphi}
E(X,\Phi_m)
:=
\sum_{\substack{q =1 \\ (q, m) = 1}}^{\infty} \left|
N^{\pm}(X, \Phi_m \Psi_{q^2})
- \sum_{\sigma \in \{ 1, \frac{5}{6} \}} \frac{X^\sigma}{\sigma}
\cdot \Res_{s = \sigma} \xi^{\pm}(s, \Phi_m \Psi_{q^2}) \right|.
\end{equation}
We will sum this error term over the functions
\[
\Phi_m = \Phi_U \Theta_{\textdivblahblah, \frac{r}{d}} \Theta_{\textdivblahblah^2, d^2} \Theta_{(1^3), t^2}
\]
appearing in \eqref{eq:Phip2},
and in this section we will prove the bound
\begin{equation}\label{eq:lc_setup_2}
\sum_{r \leq R} \sum_{t \leq T} \sum_{d\mid r}
\left|
E(X_{r,t}, \Phi_m)
 \right|
\ll X^{2/3 + \epsilon} U^{2/3} R^{2/3}.
\end{equation}
provided that $\frac{X}{2} \leq X_{r, t} \leq X$ for each $r$ and $t$.
Once this bound is established,
Theorem 1.4 is proved as follows:

We first note that for each $r$ and $t$, the inner sum over $d$
in \eqref{eq:lc_setup_2} is
the sum over the set of all $\Phi_m$ appearing in \eqref{eq:Phip2},
and by \eqref{eq:exphi} is therefore a bound on the error made
in approximating  $N_{\leq 3}^\pm(X_{r, t}, \Sigma^{r, t})$
in terms of residues of Shintani zeta functions.
We write out the analogue of \eqref{eq:sums},
with the quantities $C_p(\Sigma_p)$ and
$K_p(\Sigma_p)$ of Table \ref{table:error_exponents} coming from our
residue formula \eqref{eq:res_other}.
We use Propositions \ref{lem:quad_count} and~\ref{prop:red1}
to relate $N^{\pm}_{\leq 3}(X_{r, t}, \Sigma^{r, t})$
to $N^{\pm}_3(X_{r, t}, \Sigma^{r, t})$, and note that the
error terms in Proposition~\ref{lem:quad_count} are less than those
in Theorem~\ref{thm_lc_averaged}.
Thus we have a proof of Theorem~\ref{thm_lc_averaged},
provided that $\frac{X}{2} \leq X_{r, t} \leq X$ for each $r$ and $t$.
Finally, we extend \eqref{eq:lc_setup_2} to allow arbitrary $X_{r, t} \leq X$
by dividing the interval $[1, X]$ into $O(X^{\epsilon})$ dyadic intervals,
considering the contribution of those $(r, t)$ with $X_{r, t}$ in each interval,and then summing the resulting error terms.
The sum of the error terms is again $O(X^{2/3 + \epsilon} U^{2/3} R^{2/3})$,
finishing the proof of Theorem~\ref{thm_lc_averaged}.

To show \eqref{eq:lc_setup_2}, we rearrange the problem and prove:
\begin{proposition}\label{prop:lc_sum_of_errors}
Given parameters $R'$ and $T'$,
we have
\begin{equation}\label{eq:lc_sum_of_errors}
\sum_{\Phi_m} 
\left|
E(X_{\Phi_m},\Phi_m)
 \right|
\ll X^{2/3 + \epsilon} U^{2/3} (R')^{2/3},
\end{equation}
where $\Phi_m = \Phi_U \Theta_{\textdivblahblah, \frac{r}{d}} \Theta_{\textdivblahblah^2, d^2} \Theta_{(1^3), t^2}$ as above,
and the sum is over $(r,t,u)$ satisfying:
$r$ and $t$ range over squarefree integers coprime to $U$ and to each other, and $d$ ranges over divisors of $r$;
with $r' := r/d \in [R', 2R']$ and $t' := dt \in [T', 2T']$.
Here 
$X_{\Phi_m}$ may vary arbitrarily within a dyadic interval $\frac{X}{2} \leq X_{\Phi_m} \leq X$ for each $\Phi_m$.
\end{proposition}
To obtain \eqref{eq:lc_setup_2} from Proposition \ref{prop:lc_sum_of_errors},
we extend the ranges of $r'$ and $t'$ in the left hand side
of \eqref{eq:lc_sum_of_errors}
to the set of pairs $(r', t')$ with $r' \leq R$ and $r' t' \leq RT$
by summing the right hand side of \eqref{eq:lc_sum_of_errors}
over $O(X^{\epsilon})$ dyadic intervals.
The resulting bound is again
$O( X^{2/3 + \epsilon} U^{2/3} R^{2/3})$,
and this verifies \eqref{eq:lc_setup_2}.

Thus it is enough to prove Proposition \ref{prop:lc_sum_of_errors},
which we start now.
Theorem \ref{thm:landau_average} implies that
\begin{equation}\label{eq:lc_setup}
\sum_{\Phi_m} 
\left|
E(X_{\Phi_m},\Phi_m)
 \right|
\ll E'_1 + E'_2 + E_3',
\end{equation}
with 
\begin{align*}
E'_1 := & \ \sum_{\sigma = 1, 5/6} X^\sigma \sum_{\Phi_m}
 \sum_{\substack{q > Q}}
\left|\Res_{s = \sigma} \xi^{\pm}(s, \Phi_m \Psi_{q^2})\right|,
\\
E'_2 := & \ X^{\frac35}
\sum_{Q_1} 
\Bigg( \sum_{\Phi_m} \sum_{q \in [Q_1, 2Q_1]} \delta_1(\Phi_m \Psi_{q^2}) \Bigg)^{\frac35} 
\Bigg( \sum_{\Phi_m} \sum_{q \in [Q_1, 2Q_1]} \widehat{\delta_1}(\Phi_m \Psi_{q^2}) \Bigg)^{\frac25}
\\
= & \ X^{\frac35}
\sum_{Q_1} E_{2}''(Q_1)^{3/5} {\widehat E}_{2}''(Q_1)^{2/5},
\\
E_3' := &
\sum_{\Phi_m}
\sum_{q > Q} N^{\pm}(X, \Phi_m \Psi_{q^2}),
\end{align*}
where again $Q_1$ ranges over integer powers of $2$ less than $Q$, 
subject to the condition 
that \eqref{eq:error_smaller_2} holds, i.e., for each $Q_1 < Q$ that we have
\begin{equation}\label{eq:error_smaller_3}
{\widehat E''}_2(Q_1) \ll E''_2(Q_1) X,
\end{equation}
which we will check later.

For ease of reading we replace $r', t', R', T'$ by $r, t, R, T$ respectively.
Thus in particular
$\Phi_m = \Phi_U \Theta_{\textdivblahblah, r} \Theta_{\textdivblahblah^2, d^2} \Theta_{(1^3), t^2/d^2}$.
By Theorem~\ref{thm:fe},
the residues of $\xi^\pm(s,\Phi_m\Psi_{q^2})$
at $s=1$ and $5/6$ are 
$\ll X^\epsilon R^{-1} T^{-2} q^{-2}$
and $\ll X^\epsilon R^{-1} T^{-5/3} q^{-5/3}$ respectively, with analogous lower bounds with $X^{-\epsilon}$ in place of $X^{\epsilon}$, so that
\[
E_1' \ll X^{1+\epsilon} Q^{-1}T^{-1}+X^{5/6+\epsilon}Q^{-2/3}T^{-2/3}
\ll X^{1+\epsilon} Q^{-1}T^{-1}
\]
for $QT<X^{1/2}$, and
\begin{equation}\label{eq:e2pp_bound}
X^{-\epsilon} Q_1^{-1}T^{-1} \ll E_2''(Q_1)\ll X^{\epsilon} Q_1^{-1}T^{-1}.
\end{equation}

We can also use the {tail estimate} of Proposition \ref{prop:tail_no_eps} to bound $E_3'$: we have
\begin{align*}
E_3' = &
\sum_{r \in [R, 2R]} \sum_{t \in [T, 2T]} \sum_{d \mid t}
\sum_{q > Q} N^{\pm}(X, \Phi_m \Psi_{q^2}) \\
\ll & \
X^{\epsilon} \sum_{t \in [T, 2T]} \sum_{d \mid t}
\sum_{q > Q} N^{\pm}(X, \Phi_{Ut^2} \Psi_{q^2}) \\
\ll & \
X^{\epsilon} \sum_{t \in [T, 2T]} \sum_{d \mid t} \sum_{q > Q} X t^{-2 + \epsilon}
q^{-2 + \epsilon}
\\
\ll & \
X^{1 + \epsilon} Q^{-1} T^{-1},
\end{align*}
where in the second line $\Phi_{Ut^2}$ is the product of the local conditions
modulo $Ut^2$. Here we use the fact that any
$R$ with $0 < |\Disc(R)| < X$ satisfies $p \mid \Disc(R)$ for at most $O(X^{\epsilon})$
primes~$p$.

We claim the following average bound on
$\widehat{\delta_1}(\Phi_m \Psi_{q^2})$:
\begin{proposition}
We have
\begin{equation}\label{eq:e22}
{\widehat E}_2''(Q_1)
:=\sum_{q \in [Q_1, 2Q_1]} \widehat{\delta_1}(\Phi_m \Psi_{q^2})
\ll 
(RTUQ_1)^2 X^{\epsilon}.
\end{equation}
\end{proposition}
\begin{proof}
By the definition \eqref{eqn:def_delta1}, the claim to be proved is 
\begin{multline}\label{eqn:bound_delta_5a}
\sum_{q \in [Q, 2Q]} \sum_{r \in [R, 2R]} \sum_{t \in [T, 2T]} \sum_{d \mid t}
\sum_{|\Disc(x)| < N} |\widehat{\Phi_U}(x) \widehat{\Phi_{r}}(x) 
\widehat{\Theta_{\textnormal{div}^2, d^2}}(x) \widehat{\Theta_{(1^3), \frac{t^2}{d^2}}}(x)
\widehat{\Psi_{q^2}}(x)|\\
\ll
N
Q^{-6} 
R^{-2}
T^{-6}
U^{-2}
X^{\epsilon}.
\end{multline}
Here, and in what follows, the sum ranges over $q,r,t\in\Z$ such that
$qrtU$ is squarefree.

We expand the definition of $\Psi_{p^2}$ to be any of the functions 
$\Theta_{(1^3), p^2}$, 
$\Theta_{\textnmax, p^2}$ or $\Theta_{\textnormal{div}^2, p^2}$,
all of which
satisfy the Fourier transform bounds of Proposition \ref{prop:bound_delta_2}.
We may simplify \eqref{eqn:bound_delta_5a} by combining the $T$ and $Q$ variables. 
Since any product $qt$ with $q \in [Q, 2Q]$ and $t \in [T, 2T]$
has $O((QT)^{\epsilon})$ such factorizations, and any $t$ in the sum has at most $O(X^{\epsilon})$ divisors,
it suffices to prove that
\begin{equation}\label{eqn:bound_delta_5b}
\sum_{q \in [Q, 2Q]} \sum_{r \in [R, 2R]} 
\sum_{|\Disc(x)| < N} |\widehat{\Phi_U}(x) \widehat{\Phi_{r}}(x) 
\widehat{\Psi_{q^2}}(x)|
\ll
N
Q^{-6} 
R^{-2}
U^{-2}
X^{\epsilon}.
\end{equation}
(Our claim is initially reduced to 
proving a variation of \eqref{eqn:bound_delta_5b} with a sum over $q \in [QT, 4QT]$ on the left and $(QT)^{-6}$ in place of $Q^{-6}$
on the right; a change of variables and two applications of \eqref{eqn:bound_delta_5b} yield such a bound.)

The Fourier transform $\widehat{\Phi_m}$ is multiplicative. 
We require the following bounds on the Fourier transforms of the constituent functions
$\widehat{\Phi_p}$ for $p \mid m$ due to Mori~\cite[Theorem~1]{Mori}; see also
\cite[Proposition 6.1]{TT_L} and \cite[Theorem~11 and Corollary~12]{TT_orbital}.
 
\begin{proposition}
Let $p$ be a prime, and let $\Phi_p : V(\Z/p\Z) \rightarrow \{ 0, 1 \}$ be any of
$\Phi_{(111), p}$, $\Phi_{(21), p}$, or $\Phi_{(3), p}$.
Then 
\begin{equation}
|\widehat{\Phi_p}(x)| \ll
\begin{cases}
	1	& x=0,\\
	p^{-1}		& \text{$x$ has a triple root modulo $p$},\\
		p^{-2}		& \text{otherwise}.
\end{cases}
\end{equation}
\end{proposition}

\begin{proposition}
Let $p$ be a prime.
The function
$\Phi_{\textdivblahblah, p} \ : V(\Z/p\Z) \rightarrow \{ 0, 1 \}$
satisfies 
\begin{equation}
|\widehat{\Phi_{\textdivblahblah, p}}(x)| \ll
\begin{cases}
	p^{-1}	& x=0,\\
	p^{-2}		& x \neq 0, \ \ p \mid \Disc(x),\\
		p^{-3}		& \text{otherwise}.
\end{cases}
\end{equation}
\end{proposition}
Analogously to \eqref{eqn:total_contribution}, we therefore obtain that \eqref{eqn:bound_delta_5b} is bounded above by
\begin{equation}\label{eqn:total_contribution3}
X^{\epsilon} \sum_{c_2, c_1, d_{4c}, d_{4n}, d_3, d_2}
\frac{1}{c_2^2 c_1^3 d_{4c}^2 d_{4n}^3 d_3^4 d_2^5}
\sum_{r \in [R, 2R]}
\sum_{\substack{|\Disc(x')|  < \frac{N}{c_2^8 c_1^4 d_{4c}^6 d_{4n}^2 d_3^2} \\
d_{4n}^2 d_3 d_2^2 \mid \Disc(x')}} 
 |\widehat{\Phi_U}(x') \widehat{\Phi_{r}}(x')|,
\end{equation}
where the notation is as in \eqref{eqn:total_contribution}, and in particular the outer sum is over all choices
of the variables, squarefree and coprime to each other and to $Ur$, and the product is in $[Q, 2Q]$. 

We now unravel the inner sum by summing over all factorizations $U = u_1 u_2 u_3$
and $r = r_1 r_2 r_3$, where $u_1 r_1 \mid x$, and $u_2 r_2 \nmid x$ but
$u_2^2 r_2 \mid \Disc(x)$, and then replacing $x'$ with $x'' := \frac{x'}{u_1 r_1}$.
We thus see that the innermost sum is bounded above by
\[
X^{\epsilon}
\sum_{u_1 u_2 u_3 = U} 
\sum_{r_1 r_2 r_3 = r}
\sum_{\substack{|\Disc(x'')|  < \frac{N}{c_2^8 c_1^4 d_{4c}^6 d_{4n}^2 d_3^2 u_1^4 r_1^4} \\
d_{4n}^2 d_3  d_2^2 u_2^2 r_2 \mid \Disc(x'')}} (u_2 r_1)^{-1} (u_3 r_2)^{-2} r_3^{-3},
\]
so that the total sum is bounded above by
\begin{equation}\label{eqn:total_contribution2a}
X^{\epsilon} 
\sum_{\substack{c_2, c_1, d_{4c}, d_{4n}, d_3, d_2 \\ u_1, u_2, u_3, r_1, r_2, r_3}}
\frac{1}{c_2^2 c_1^3 d_{4c}^2 d_{4n}^3 d_3^4 d_2^5 
u_2 r_1 u_3^2 r_2^2 r_3^3}
\sum_{\substack{|\Disc(x'')|  < \frac{N}{c_2^8 c_1^4 d_{4c}^6 d_{4n}^2 d_3^2 u_1^4 r_1^4} \\
d_{4n}^2 d_3  d_2^2 u_2^2 r_2 \mid \Disc(x'')}}
1.
\end{equation}
As before, for small $N$ we sum the variable $d_3$ over dyadic intervals $[D_3, 2D_3]$, and we similarly sum 
$r_2$ over dyadic intervals $[R_2, 2R_2]$, and use the fact that there are at most
$O(N^{\epsilon}) = O(X^\epsilon)$ pairs $d_3$ and $r_2$ with $d_3 r_2 \mid \Disc(x'')$ for any fixed $x''$ with $|\Disc(x'')| \ll N$.
The above sum therefore reduces to
\begin{align*}
& 
X^\epsilon 
\sum_{D_3, R_2} 
\sum_{\substack{c_2, c_1, d_{4c}, d_{4n}, d_2 \\ u_1, u_2, u_3, r_1, r_3}}
\frac{1}{c_2^2 c_1^3 d_{4c}^2 d_{4n}^3 D_3^4 d_2^5 
u_2 r_1 u_3^2 R_2^2 r_3^3} 
\cdot \frac{N}{c_2^8 c_1^4 d_{4c}^6 d_{4n}^2 D_3^2 u_1^4 r_1^4} \cdot
\frac{1}{d_{4n}^2  d_2^2 u_2^2 } \\
= \ &
X^\epsilon 
\sum_{D_3, R_2} 
\sum_{\substack{c_2, c_1, d_{4c}, d_{4n}, d_2 \\ u_1, u_2, u_3, r_1, r_3}}
\frac{1}{c_2^{10} c_1^7 d_{4c}^8 d_{4n}^7 D_3^6 d_2^7}
\cdot
\frac{1}{u_1^4 u_2^3 u_3^2 r_1^5 R_2^2 r_3^3} \cdot N \\
\ll \ &
X^\epsilon U^{-2} R^{-2}
\sum_{D_3} 
\sum_{c_2, c_1, d_{4c}, d_{4n}, d_2}
\frac{1}{c_2^{10} c_1^7 d_{4c}^8 d_{4n}^7 D_3^6 d_2^7} \cdot N \\
\ll \ &
X^\epsilon Q^{-6} U^{-2} R^{-2} N,
\end{align*}
as desired, establishing \eqref{eq:e22}. For large $N$, as before we apply Proposition \ref{prop:uniformity} to obtain the same bound.
\end{proof}

We now complete the proof of Proposition \ref{prop:lc_sum_of_errors}.
Collecting all of our error terms, we have
\begin{align} \label{eq:error_collected}
E'_1 + E'_2 + E'_3 \ll & \ X^{\epsilon} \Bigl( X Q^{-1} T^{-1} + X^{3/5} \sum_{Q_1} \bigl( Q_1^{-1} T^{-1} \bigr)^{3/5}
(R^2 T^2 U^2 Q_1^2)^{2/5} + XQ^{-1}T^{-1} \Bigr) \\ \nonumber
\ll & \ X^{\epsilon} \bigl( X^{3/5} Q^{1/5} T^{1/5} R^{4/5} U^{4/5}
+ XQ^{-1}T^{-1} \bigr).
\end{align}
Suppose $U^2 R^2 T^3<X^{1 - \nu}$ for some small $\nu > 0$.
Optimizing the bound,
we choose $Q = X^{1/3 - \epsilon} T^{-1} R^{-2/3} U^{-2/3}$.
The lower bound in \eqref{eq:e2pp_bound}
shows that \eqref{eq:error_smaller_3} is satisfied, and so we have
\[
E'_1 + E'_2 + E'_3 \ll X^{2/3 + \epsilon} U^{2/3}  R^{2/3},
\]
finishing the proof in this case.

If instead $U^2 R^2 T^3\geq X^{1 - \nu}$, we have
\[
\sum_{\Phi_m} 
\left|
E(X_{\Phi_m},\Phi_m)
 \right|
\leq E'_1 + E_3'
\]
with $Q = 0.9$, directly from the definition of $E(X_{\Phi_m},\Phi_m)$.
The bounds on $E_1'$ and $E_3'$ proved above are true for any $Q$,
and we have $E_1'+E_3'\ll X^{1+\epsilon'}T^{-1}$ for any $\epsilon'>0$.
This is bounded above by $X^{2/3 + \epsilon} U^{2/3} R^{2/3}$
in the present case. This finishes the proof.

\begin{remark}\label{remark:other_split}
As stated in Theorem \ref{thm_lc}, our method can handle essentially arbitrary local conditions. %, at least in principle.
Given any $\GL_2(\Z/M'\Z)$-invariant
function $\Phi_{M'} : V(\Z/M'\Z) \rightarrow \C$ with $(M', M) = 1$ and $|\Phi_{M'}(x)| \leq 1$ for all $x$, we may replace $\Phi_M$ and $\Phi_m$ by $\Phi_M \Phi_{M'}$ 
and $\Phi_m \Phi_{M'}$
throughout the~proof. 

In particular, for $p^a \mid\mid M'$ one may choose  $\Phi_{p^a} : V(\Z/p^a \Z) \rightarrow \{ 0, 1\}$ corresponding to a fixed ramified cubic
algebra over $\Q_p$,
the possibilities for which were described below
Theorem~\ref{thm_lc_averaged}. Again insisting on maximality at $p$,
we may always take $a = 2$ when $p > 3$.

The assumption that
$U^2 R^2 T^3 < X^{1 - \epsilon}$ may be strengthened to $U^2 R^2 T^3 M'^4 < X^{1 - \epsilon}$.
 One must
then obtain residue formulas for the Shintani zeta function (carried out in some cases in \cite[Section 8]{TT_L}),
check condition \eqref{eq:cond_lower_error}, and prove an analogue of 
Proposition \ref{lem:quad_count}, showing that the contribution of reducible rings is as expected. 
The remaining arguments are then the same; in \eqref{eq:error_collected}, one
multiplies the expression $R^2 T^2 U^2 Q_1^2$ by $M'^4$, or by a smaller power of $M'$ given suitable bounds on $\widehat{\Phi_{M'}}(x)$.

One thus obtains the counting function for all cubic orders
satisfying the local conditions prescribed by $\Phi_M$ and $\Phi_{M'}$, and which are maximal except as prescribed by $\Phi_{M'}$.

\end{remark}

\section{Proof of Theorem \ref{thm_rc} using a discriminant-reducing identity}\label{sec:disc_reduce_1}

We now describe a ``discriminant-reducing identity'' for nonmaximal cubic rings. It is essentially
equivalent to an identity in \cite[Section 9.1]{BST}, but we give a formulation and proof 
in the language of Shintani zeta functions. We then apply our identity to give a second proof of
Theorem \ref{thm_rc}, but with the $O(X^{2/3 + \epsilon})$ error term of Section \ref{sec:direct_proof}
improved to $O(X^{2/3} (\log X)^{\alpha})$ with any $\alpha > -\frac{1}{2} + \frac{5}{3^{3/5}}$.
Note that our identity allows for the presence of unrelated local conditions, so that
we could likely also obtain another proof of Theorems \ref{thm_lc} and \ref{thm_lc_averaged}
in the case where $\Sigma_p = A_p$ at every ordinary prime.

In this section, for each squarefree integer $q$, we write $\Psi_{q^2} : V(\Z/q^2 \Z) \rightarrow \{ 0, 1 \}$ for 
the characteristic function
of those 
$x \in V(\Z/q^2 \Z)$ that are nonmaximal at each prime divisor of~$q$.
We also write $\eta_q : V(\Z/q\Z) \rightarrow \Z$ for the function counting the number of roots of $x \in V(\Z/q\Z)$ in $\mathbb{P}^1({\Z/q\Z})$. 

\begin{proposition}\label{thm_shintani_bst}
For any
$\GL_2(\Z/M\Z)$-invariant function
$\Phi_M : V(\Z/M\Z) \rightarrow \C$ with $(M, q) = 1$, we have
\begin{equation}\label{eqn_shintani_bst_2}
\xi^{\pm}(s, \Psi_{q^2} \otimes \Phi_M) = 
q^{-4s} \sum_{k l m = q} \mu(l) k^{2s} \xi^{\pm}(s, \eta_{kl} \otimes \Phi_M).
\end{equation}
\end{proposition}
\begin{proof}
By induction on the number of prime factors of $q$, it suffices to assume that $q = p$ is prime.
For convenience write $\Gamma := \GL_2(\Z)$, and define
\begin{align*}
X_p &
=\{(x_1,x_2,x_3,x_4)\in V(\Z):  x_3 \equiv 0 \!\!\!\!\pmod{p}, \,\,x_4 \equiv 0 \!\!\!\!\pmod{p^2}\},
\\
Y^p &
=\{(x_1,x_2,x_3,x_4)\in V(\Z):  x_1 \equiv 0 \!\!\!\!\pmod{p}\},
\\
B_p &
=\bigl\{\twtw{\alpha}{\beta}{\gamma}{\delta} \in \Gamma: \gamma \equiv 0 \!\!\!\!\pmod{p}\bigr\}, \\
B^p &
=\bigl\{\twtw{\alpha}{\beta}{\gamma}{\delta} \in \Gamma: \beta \equiv 0 \!\!\!\!\pmod{p}\bigr\}.
\end{align*}
Then $B_p$ and $B^p$
preserve $X_p$ and $Y^p$ respectively.
The action of the matrix 
$\twtw{1}{0}{0}{1/p} \in\GL_2(\Q)$
induces a bijection $\phi : X_p \rightarrow Y^p$, 
such that if $gx = x'$ in $X_p$ with
$g = \twtw{\alpha}{\beta}{\gamma}{\delta} \in B_p$,
then $\widetilde{g} \phi(x) = \phi(x')$ with
$\widetilde{g} = \twtw{\alpha}{\beta p}{\gamma/p}{\delta} \in B^p$. 
Writing $B_{p, x}$ and $B^p_x$ for $\Stab_{B_p}(x)$ and $\Stab_{B^p}(x)$ respectively,
$\phi$ thus induces
the equality
\begin{equation}\label{eqn_shintani_bst_3}
\sum_{\substack{x \in B_p \backslash X_p\\ \pm\Disc(x)>0}}
\frac{1}{|B_{p,x}|} \Phi_M(x) |\Disc(x)|^{-s}
=
p^{-2s} \sum_{\substack{x \in B^p \backslash Y^p\\ \pm\Disc(x)>0}} \frac{1}{|B^p_x|} \Phi_M(x) |\Disc(x)|^{-s}.
\end{equation}

We claim that 
the right side of \eqref{eqn_shintani_bst_3} is
$p^{-2s} \xi^{\pm}(s, \eta_p \otimes \Phi_M)$; it suffices to show
that for each $x\in V(\Z)$ that
\begin{equation}\label{eq:eta-identity}
\eta_p(x)=\sum_{y\in B^p\backslash (\Gamma x\cap Y^p)}
	[\Gamma_y:B^p_y].
\end{equation}
Recall that there is a transitive right action of $\GL_2(\Z)$ on
$\P^1(\Z/p\Z)$, given by matrix multiplication when we represent a 
point $(t_0 : t_1) \in \P^1$ as a row vector.
With this action, $B^p$ is the stabilizer of $(1:0)\in \P^1(\Z/p\Z)$.
We identify $B^p\backslash \Gamma$ with
$\P^1(\Z/p\Z)$ via $[g]\mapsto(1:0)g$, so that
\[
\eta_p(x)=\#\{[g]\in B^p\backslash \Gamma \mid gx\in Y^p\}.
\]
Consider the map
\[
\{[g]\in B^p\backslash \Gamma\mid gx\in Y^p\}
\longrightarrow B^p\backslash (\Gamma x\cap Y^p),
\qquad
[g]\longmapsto [gx].
\]
It is surjective by construction. Writing $y = gx$, we have $B^p y = B^p g'x$ if and only if $y = h g' g^{-1} y$ for some $h \in B^p$, i.e., 
if and only if $g' \in B^p \Gamma_y g$.
Therefore, the fibers each have size
\[
|B^p \backslash B^p \Gamma_y g| =
|B^p \backslash B^p \Gamma_y| = |B^p \cap \Gamma_y \backslash \Gamma_y| = [\Gamma_y : B^p_y],
\]
proving \eqref{eq:eta-identity}.

To conclude, we show that the left side of \eqref{eqn_shintani_bst_3} is 
\begin{equation}\label{eqn_shintani_bst_4}
\xi^\pm(s, \Psi_{p^2} \otimes \Phi_M) - p^{-4s}  \xi^\pm(s, \Phi_M) + p^{-4s} \xi^\pm(s, \eta_p \otimes \Phi_M).
\end{equation}

For those $x \in X_p$ not in $p V(\Z)$, it is readily checked that $\Gamma_x = B_{p, x}$, so that the contribution of such $x$
is equal to
$\xi^\pm(s, \Psi_{p^2} \otimes \Phi_M) - p^{-4s} \xi^\pm(s, \Phi_M)$
by Proposition \ref{def_up}.
The contribution from those $x \in X_p\cap p V(\Z)$ is
\begin{equation}\label{eqn_shintani_bst_5}
p^{-4s} \sum_{\substack{x \in B_p \backslash Y_p\\ \pm\Disc(x)>0}} \frac{1}{|B_{p,x}|} \Phi_M(x) |\Disc(x)|^{-s},
\end{equation}
where $Y_p=\{(x_1,x_2,x_3,x_4)\in V(\Z): \ x_4 \equiv 0 \pmod{p}\}$.
Let $\gamma=
\left(\begin{smallmatrix}0&1\\1&0\end{smallmatrix}\right)\in\Gamma$.
Then since $Y^p=\gamma Y_p$ and $B^p=\gamma B_p\gamma^{-1}$,
 the bijection $Y_p\ni x\mapsto\gamma x\in Y^p$
induces the bijection $B_p\backslash Y_p\rightarrow B^p\backslash Y^p$,
and the sum in \eqref{eqn_shintani_bst_5} coincides with the sum
in the right side of \eqref{eqn_shintani_bst_3}.
Thus the contribution \eqref{eqn_shintani_bst_5}
is $p^{-4s}\xi^\pm(s, \eta_p \otimes \Phi_M)$. 
\end{proof}

To give another proof of Theorem \ref{thm_rc}, we again apply Landau's method and treat the error terms
$E_1$, $E_2$, and $E_3$ from \eqref{eq:main_term}. To avoid $O(X^{\epsilon})$ factors in the error terms,
we recall several standard bounds from analytic number theory.
For any fixed $C\geq1$, we have
\begin{equation}\label{eq:IK}
\sum_{q \leq Q} \mu^2(q) C^{\omega(q)} \ll Q (\log Q)^{C - 1},
\end{equation}
which can be proved (for example) as in \cite[p.~24]{IK}. We also have
\begin{equation}\label{eq:IK2}
\sum_{q > Q} \mu^2(q) \frac{C^{\omega(q)}}{q^\alpha} \ll \frac{(\log Q)^{C - 1}}{Q^{\alpha - 1}}
\end{equation}
for each fixed $\alpha > 1$ and $C\geq1$, as can be seen by subdividing the interval $(Q, \infty)$ into dyadic subintervals and applying 
\eqref{eq:IK} to each. Finally, we have
\begin{equation}\label{eq:gronwall}
\prod_{p \mid n} \Big( 1 + \frac{1}{p} \Big) \ll \log \log n
\end{equation}
as a consequence of Gronwall's classical bound $\sigma(n) \ll n \log \log n$ \cite{gronwall}.

We now return to \eqref{eq:main_term}
with $M=1$.
By \eqref{eq:res_nmax} and \eqref{eq:IK2}, we obtain
\[
E_1 \ll X \sum_{q > Q} \prod_{p \mid q} \big( 2 p^{-2} \big) + X^{5/6} \sum_{q > Q} \prod_{p \mid q} \big( 2 p^{-5/3} \big)
\ll 
\frac{X (\log Q)}{Q} + \frac{X^{5/6} (\log Q)}{Q^{2/3}}.
\]
For $E_3$,
by using Proposition \ref{prop:tail_nonmax}
in place of Proposition \ref{prop:tail_no_eps},
we have
\[
E_3 \ll X \sum_{q > Q} 3^{\omega(q)} q^{-2} \ll \frac{X (\log Q)^2}{Q}.
\]

To treat the $E_2$ error term,
we apply Theorem \ref{thm:landau} (instead of 
Theorem \ref{thm:landau_average})
to each Shintani zeta function
in the summation in \eqref{eqn_shintani_bst_2} with $M = 1$.
This gives
\begin{equation}
E_2 := X^{3/5} 
\sum_{q \leq Q} 
\sum_{k l m = q}
(q^{-4} k^2)^{3/5} \big(\delta_1(\eta_{kl}) \big)^{3/5}  \big(\widehat{\delta_1}(\eta_{kl}) \big)^{2/5}.
\end{equation}
We have $\delta_1(\eta_{kl}) \ll 2^{\omega(kl)}$ by \eqref{eq:res_roots}, and by definition we have
\[
\widehat{\delta_1}(\eta_{kl}) 
\leq (kl)^4 \sup_N \frac{1}{N}\sum_{\alpha \in\{\pm\}} \sum_{n < N} a^{\alpha}(|\widehat{\eta_{kl}}|, n).
\]
We will prove, for each 
squarefree $d$ and $N > 1$, that
\begin{equation}\label{eqn:dyadic1}
\sum_{\alpha \in\{\pm\}} \sum_{n < N} a^{\alpha}(|\widehat{\eta_d}|, n) \ll N d^{-3} 6^{\omega(d)} \log \log d.
\end{equation}
Granting this for now, we have
\begin{align}
E_2 & \ll X^{3/5} \sum_{q \leq Q} 
\sum_{k l m = q}
(q^{-4} k^2)^{3/5} \big( 2^{\omega(kl)} \big)^{3/5} \big(kl 6^{\omega(kl)} \log \log(kl) \big)^{2/5} \label{eqn:EQ}
\\
&
\leq X^{3/5} \sum_{q \leq Q} q^{-4/5} 2^{\frac{3}{5} \omega(q)} 6^{\frac{2}{5} \omega(q)} (\log \log q)^{2/5}
\sum_{k l m = q}l^{-6/5}m^{-8/5}
\\
& \ll X^{3/5} \sum_{q \leq Q} q^{-4/5} 2^{\frac{3}{5} \omega(q)} 6^{\frac{2}{5} \omega(q)} (\log \log q)^{2/5}
\nonumber
\\
& \ll X^{3/5} Q^{1/5} (\log Q)^{2 \cdot 3^{2/5} - 1} (\log \log Q)^{2/5}, \label{eqn:EQ2}
\end{align}
so that choosing $Q = X^{1/3} (\log X)^\beta$
with $\beta = \frac56 ( 3 - 2\cdot 3^{2/5} ) = -.086\dots$ yields
$E_1 + E_2 + E_3 \ll X^{2/3} (\log X)^\alpha$
for any $\alpha>2-\beta=2.086\dots$,
which implies Theorem \ref{thm_rc}.

It remains to prove \eqref{eqn:dyadic1}. For this, we first recall the following formula for $\widehat{\eta_d}(x)$, first (essentially) proved
by Mori \cite{Mori}, with a second simpler proof given by the second and third authors in
\cite[Proposition 1]{TT_orbital}:

\begin{proposition}\label{thm:Phip}
The function $\widehat{\eta_d}(x)$ is multiplicative in $d$, and for a prime $p \neq 3$ we have
\begin{equation}\label{eq:Psi_fourier}
\widehat{\eta_p}(x)=
\begin{cases}
	1+p^{-1}	& x=0,\\
	p^{-1}		& \text{$x$ has a triple root modulo $p$},\\
		0		& \text{otherwise},
\end{cases}
\end{equation}
where we regard $x$ as a point in $V_{\Z/p\Z}$.

\end{proposition}
We now prove \eqref{eqn:dyadic1}.
Replacing $d$ by $d/3$ if necessary, we may assume that $d$ is coprime to~3.
We then sum over all factorizations $d = d_1 d_2$, and consider the contribution
of 
those $x$ which reduce to zero in $\Z/d_1\Z$ and which 
have a triple root in $\Z/p\Z$ for each $p\mid d_2$. 

By Proposition \ref{prop:tail_no_eps}, the number of such 
$x$ with $|\Disc(x)|<N$
is \smash{$\ll \frac{N 6^{\omega(d_2)}}{d_1^4 d_2^2}$}, and so
\begin{align*}
\sum_{|\Disc(x)| < N} |\widehat{\eta}_d(x)|
\ll & \sum_{d = d_1 d_2} \bigg(  \prod_{p | d_1} (1 + p^{-1})  \prod_{p | d_2} p^{-1} \bigg) \frac{N 6^{\omega(d_2)}}{d_1^4 d_2^2} \\
= & \ N d^{-3} 6^{\omega(d)} \prod_{p \mid d} \left( 1 + \frac{1}{6} \cdot \frac{1 + p^{-1}}{p} \right) \\
\ll & \ N d^{-3} 6^{\omega(d)} \log \log d,
\end{align*}
completing the proof of Theorem \ref{thm_rc}.

\section{Counting quadratic fields}\label{appendix:count_quadratic}

In this section, we prove Proposition \ref{lem:quad_count} which estimates the number
of quadratic fields of bounded discriminant
satisfying a prescribed set of local conditions.
As we explain, this amounts to counting
squarefree integers in $[1, X]$ in suitable arithmetic
progressions. 

We begin by formulating our results precisely.
Recall that a {\itshape local specification} $\Sigma_p$ at a prime~$p$ is a subset 
of the cubic \'etale algebras over $\Q_p$;
we say that a cubic field $F$ satisfies
the local specification $\Sigma_p$ if $F \otimes_\Q \Q_p \in \Sigma_p$,
and that a quadratic field $F_2$ satisfies 
the local specification $\Sigma_p$ if the algebra $F_2 \times \Q$ does.

For quadratic fields, we may assume without loss of generality that $\Sigma_p$ contains
algebras only of the form 
$(F_2)_p \times \Q_p$, where $(F_2)_p$ is either $\Q_p \times \Q_p$
or a quadratic field
extension of $\Q_p$. In this section, we abuse notation and also write $\Sigma_p$ for the set of 
$(F_2)_p$ with $(F_2)_p \times \Q_p \in \Sigma_p$.

Let $p\neq2$ and let $l$ be an arbitrary non-square element
in $\Z_p^\times$. Then it is immediately checked that the left column of the following table
enumerates the possibilities for $(F_2)_p$:

\begin{center}
\begin{tabular}{| c | c| c |}\hline
$(F_2)_p$ & Condition on $d = \Disc(F_2)$ & Density $c_p$ \\ \hline \hline
$\Q_p \times \Q_p$ & $(\frac{d}{p})= 1$ & 
$\frac{1}{2} (1 + p^{-1})^{-1}$  \\[.025in] \hline
$\Q_p(\sqrt{l})$ &  
$(\frac{d}{p})= -1$  &$\frac{1}{2} (1 + p^{-1})^{-1}$ \\[.025in]\hline
$\Q_p(\sqrt{p})$ &  
$(\frac{d/p}{p})= 1$ 
& $\frac{1}{2} (p + 1)^{-1}$  \\[.025in]\hline
$\Q_p(\sqrt{lp})$ & 
$(\frac{d/p}{p})= -1$ & $\frac{1}{2} (p + 1)^{-1}$ \\ \hline
\end{tabular}
\end{center}
For each $(F_2)_p$, we write $\varepsilon_p \in \{\pm 1\}$ for
the value of the Legendre symbol in the second column.
Note that $F_2 \otimes_\Q \Q_p$ is determined by the
value of $\Disc(F_2) \!\pmod{p^2}$, and by
$\Disc(F_2) \!\pmod{p}$ when $F_2$ is unramified.

If $p = 2$, then there are eight possibilities for $(F_2)_2$;
with the exception of the first row, the following table lists a generating
polynomial for $(F_2)_2$:
\begin{center}
\begin{tabular}{ |c | c| c|}\hline
$(F_2)_2$ & Condition on $d = \Disc(F_2)$ & Density $c_2$ \\ \hline\hline
$\Q_2 \times \Q_2$ & $d \equiv 1 \pmod{8}$ & 
$1/3$  \\[.025in]\hline
$x^2 - x + 1$ &  
$d \equiv 5 \pmod{8}$ &
$1/3$ 
\\[.025in]\hline
$x^2 + 2x \pm 2$ &  
$d \equiv 4 \mp 8 \pmod{32}$
& $1/12$  \\[.025in]\hline
$x^2 + a \ (a = \pm 2, \pm 6)$ & 
$d \equiv -4a \pmod{64}$ 
& $1/24$ \\ \hline
\end{tabular}
\end{center}
To verify this, one checks that: the polynomials
in the left column satisfy the discriminant
conditions in the middle column; these 
conditions are mutually
exclusive; 
any $\alpha \in \Q_2^{\times}$ is one of the $d$ listed above
times a nonzero $2$-adic square;
and these conditions cover all residue
classes (mod $64$) that can be discriminants
of quadratic fields.

Thus the quadratic fields enumerated by $N_2^\pm(X, \Sigma)$ are in bijection with
fundamental discriminants that satisfy appropriate congruence conditions.
To distinguish among the possibilities for $F_2 \otimes_\Q \Q_2$,
we assume a congruence condition of the form $\Disc(F_2) \equiv a \ (\textmod \ 64)$.
Hereafter we assume that $\Sigma$ is a collection of local specifications over odd primes only,
and write $N_2^{\pm}(X; a, \Sigma)$ for the count of fields counted by $N_2^\pm(X, \Sigma)$
with $\Disc(F_2) \equiv a \ (\textmod \ 64)$.

We assume that 
for sufficiently large primes $p$, the local specification 
$\Sigma_p$ consists of all quadratic \'etale $\Q_p$-algebras,
and hence imposes no condition on $F_2$. At each remaining odd prime~$p$, we assume that 
either (1) $\Sigma_p$ consists of a single algebra, or
(2) $\Sigma_p = \{ \Q_p(\sqrt{p}), \Q_p(\sqrt{lp}) \}$,
in which case $F_2$ satisfies the local specification $\Sigma_p$ if and only if $p$ ramifies in $F_2$.
For $\Sigma=(\Sigma_p)_p$,
we write: $u$ for the product of primes $p$ where $\Sigma_p$
consists of a single unramified algebra;
$s$ for the product of primes~$p$ where $\Sigma_p$ consists of a single ramified algebra;
and $r$ for the product of primes $p$ where
$\Sigma_p = \{ \Q_p(\sqrt{p}), \Q_p(\sqrt{lp}) \}$.
Then 
\begin{equation}\label{eq:quad-counting-interpretation}
N_2^\pm(X;a,\Sigma)
=
\#
\left\{
	d\in\Z
	\ \vrule \ 
\begin{array}{l}
	0<\pm d<X\\
	d\equiv a\!\!\!\!\pmod{64}\\
	\text{$p\nmid d$, $(\frac dp)=\varepsilon_p$
					for all $p\mid u$}\\	
	\text{$p\parallel d$, $(\frac{d/p}p)=\varepsilon_p$
					for all $p\mid s$}\\	
	\text{$p\parallel d$				for all $p\mid r$}\\	
	\text{$d$ is not divisible by $p^2$ for any $p\nmid 2usr$}\\
\end{array}
\right\},
\end{equation}
where 
for each $p\mid us$,
$\varepsilon_p\in\{\pm1\}$
is chosen depending on $\Sigma_p$.
To count this, define
\begin{align*}
N^\pm(X;a,\Sigma;q)
:= & \
\#
\left\{
	d\in\Z
	\ \vrule \ 
\begin{array}{l}
	0<\pm d<X\\
	d\equiv a\!\!\!\!\pmod{64}\\
	\text{$d$ is divisible by $srq^2$ and $(d/r, r) = 1$}\\
	p \nmid d, \text{$(\frac dp)=\varepsilon_p$
					for all $p\mid u$}\\	
	p \parallel d, \text{$(\frac{d/p}p)=\varepsilon_p$
					for all $p\mid s$}\\	
\end{array}
\right\}
\end{align*}
%where $\varepsilon = (\varepsilon_p)_{p \mid us}$,
for each squarefree integer $q$ coprime to $2sur$.
Then by inclusion-exclusion, we have
\begin{equation}\label{eq:quad-inc-exc}
N_2^\pm(X;a,\Sigma)
=\sum_q
\mu(q)N^\pm(X;a,\Sigma;q),
\end{equation}
where $q$ runs through all squarefree integers coprime to $2usr$.
\begin{lemma}
We have
\begin{align} \label{eq:quad_ie_pv}
N^\pm(X;a,\Sigma;q)
&=
\frac{1}{64}
%\cdot 
\prod_{p \mid us} \frac{1 - p^{-1}}{2}
%\cdot
\prod_{p \mid r} \big(1 - p^{-1}\big)
%\cdot
\frac{X}{srq^2}
+
O\big(2^{\omega(r)} \sqrt{us} \log(us)\big), \\ \label{eq:quad_ie_trivial}
N^\pm(X;a,\Sigma;q)
&=O\left(\frac{X}{srq^2}\right).
\end{align}
\end{lemma}
\begin{proof}
The estimate \eqref{eq:quad_ie_trivial} is immediate. To prove \eqref{eq:quad_ie_pv},
note that 
by definition we have
\begin{equation*}
N^\pm(X;a,\Sigma;q)
=  \
\#
\left\{
	n\in\Z
	\ \vrule \ 
\begin{array}{l}
	0<\pm n<X/srq^2\\
	n\equiv a' \!\!\!\!\pmod{64}\\
	\text{$(\frac np)=\varepsilon_p'$
					for all $p\mid us$} \\ (n, r) = 1\\	
\end{array}
\right\},
\end{equation*}
where $a' := a(sr)^{-1}q^{-2} \!\pmod{64}$ and
\begin{equation}
\varepsilon_p' :=
\begin{cases}
(\frac {rs}{p})
\varepsilon_p
& \textnormal{ if } p \mid u, \\ 
(\frac{rs/p}p)
\varepsilon_p
& \textnormal{ if } p \mid s.
\end{cases}
\end{equation}
Expanding the conditions on $n$ in terms
of Dirichlet characters, and writing
$a' = bb'$ with $b := (a', 64)$, we obtain 
\begin{align*}
N^\pm(X;a,\Sigma;q)
= &
\sum_{\substack{0< n < \frac{X}{bsrq^2} \\ (n, 2usr) = 1}}
\frac{b}{32} 
\sum_{\chi_2 \,(\textmod{\,\frac{64}{b})}} \overline{\chi_2}(b')
\chi_2(\pm n)
\prod_{p \mid us}
\bigg(\frac{1 + \varepsilon_p' \chi_p(\pm n b)}{2}
\bigg) \\
= &\,
\frac{b}{2^{5 + \omega(us)}}\!\!\!\!\!
\sum_{\chi_2 \,(\textmod{\,\frac{64}{b})}} \!\!\!\!\!
\overline{\chi_2}(b')
\chi_2(\pm 1)
\sum_{m \mid us}
\varepsilon_m' \chi_m(\pm b)
\!\!\!\!\!\sum_{0< n < \frac{X}{bsrq^2}}
\!\!\!\!\!\chi_2(n) \chi_m(n)
\chi_{0, usr/m}(n)
\end{align*}
where 
$\varepsilon_m' := \prod_{p \mid m} \varepsilon_p'$,
$\chi_m := \prod_{p \mid m} \big(\frac{ \cdot }{ p } \big),$ $\chi_2$ runs over all characters $($mod ${\frac{64}{b}})$, and $\chi_{0, usr/m}$ is the principal
character $(\textmod \ usr/m)$. 

By the P\'olya-Vinogradov inequality (for imprimitive characters; 
see, e.g., \cite[Chapter~23]{davenport_book}), the
innermost sum is $\ll 2^{\omega(usr)} \sqrt{us} \log(us)$ except
when $\chi_2$ is principal and
$m = 1$.
Therefore, the expression above simplifies
to
\begin{align*}
N^\pm(X;a,\Sigma;q)
& = 
O\big(2^{\omega(r)} \sqrt{us} \log(us)\big)
+ \frac{b}{2^{5 + \omega(us)}}
\sum_{\substack{0< n < X/bsrq^2 \\ (n, 2usr) = 1}}
1 \\
& = 
O\big(2^{\omega(r)} \sqrt{us} \log(us) \big)
+ \frac{b}{2^{5 + \omega(us)}}
\cdot \frac{ \phi(2usr)}{2usr} \cdot
\frac{X}{bsrq^2}
\\
& = 
O\big(2^{\omega(r)} \sqrt{us} \log(us)\big)
+ \frac{1}{64}
\prod_{p \mid us} \frac{1 - p^{-1}}{2}  \prod_{p \mid r} \big( 1 - p^{-1} \big)
\frac{X}{srq^2}.
\end{align*}
Note that $\#\{0<n<X\mid (n,r)=1\}=\frac{\phi(r)}rX+O(2^{\omega(r)})$.
\end{proof}

\begin{proof}[Proof of Proposition~$\ref{lem:quad_count}$] 
We decompose $N^{\pm}_2(X, \Sigma)$ into 
$O(U^{\epsilon})$ counts $N^{\pm}_2(X; a, \Sigma)$, where $\Sigma$
satisfies the assumptions stated before
\eqref{eq:quad-counting-interpretation}.
By \eqref{eq:quad-inc-exc}, we then have for any $Q$ that
{\small
\begin{align*}
N_2^\pm(X; a, \Sigma) & = 
\frac{1}{64}
\prod_{p \mid us} \frac{1 - p^{-1}}{2}
\prod_{p \mid r} \big( 1 - p^{-1} \big) 
\!\!\!\!\!\!\sum_{\substack{q \leq Q \\ (q, 2usr) = 1}}
\!\!\!\!\!\frac{\mu(q)}{q^2}
\frac{X}{sr}
+
O\big(Q \cdot 2^{\omega(r)} \sqrt{us} \log(us)\big)
+ \sum_{q > Q} O\Bigl( \frac{X}{srq^2} \Bigr) \\
& = 
\frac{1}{64}
\prod_{p \mid us} \frac{1 - p^{-1}}{2}
\prod_{p \mid r} \big( 1 - p^{-1} \big)
\!\prod_{p \nmid 2usr}\! \big(1 - p^{-2} \big)
\frac{X}{sr}
+
O\bigg(Q \cdot 2^{\omega(r)} \sqrt{us} \log(us) + \frac{X}{srQ} \bigg).
\end{align*}
}Taking $Q = X^{1/2} u^{-1/4} s^{-3/4} r^{-1/2}$, we obtain an error term of 
$O(X^{1/2} u^{1/4} s^{-1/4} r^{-1/2 + \epsilon} \log(us))$. 
This yields Proposition \ref{lem:quad_count}, with 
the densities $c_p$  
corresponding to the main term above.
\end{proof}

\section*{Acknowledgments}
The idea for this paper was conceived
at the Introductory Workshop on Arithmetic Statistics at MSRI.
We would like to thank the workshop organizers (Barry Mazur, Carl Pomerance, and Mike Rubinstein) 
for bringing us together, and for an outstanding conference. We then made further progress at the workshop on
Recent Developments in Analytic Number Theory, a short six years later. We would like to thank MSRI twice over for bringing us
together, as well as the National Science Foundation
(via Grants DMS-0932078 and DMS-1440140) for its financial support of MSRI.

We would like to thank Daniel Fiorilli, Kevin McGown, Harsh Mehta, Evan O'Dorney, Arul Shankar, Keiju Sono, and an anonymous referee for helpful comments on this paper in particular. 
Most of all, we would also like to thank the very many other researchers with whom we have
enjoyed stimulating conversations on the Davenport--Heilbronn theorems.

M.B.\ was partially supported by a Simons
Investigator Grant and NSF Grant DMS-1001828; 
T.T.\ was supported by the JSPS, KAKENHI Grants JP24654005, JP25707002, JP16K13747 and JP17H02835, and JSPS Postdoctoral Fellowship for Research Abroad.
F.T.\ was partially supported by the National Science Foundation under Grants
No. DMS-0802967 and DMS-1201330, by the National Security Agency under Grant
H98230-16-1-0051, and by grants from the Simons Foundation (Nos.~563234 and~586594).

\section*{Legal Disclaimers}

{\small
No data was generated or analyzed as part of the writing of this paper. 
  
\medskip  
On behalf of all authors, the corresponding author states that there is no conflict of interest.  
  }

\bibliographystyle{alpha}
\bibliography{btt}  
\end{document}